\documentclass[11pt]{amsart}
\usepackage{amsmath, amsfonts, amsthm, amssymb, multicol, mathtools, dsfont, verbatim}
\usepackage{graphicx}
\usepackage{float, hyperref}
\usepackage{scalerel,stackengine, subcaption}
\usepackage[usenames,dvipsnames]{xcolor}
\usepackage{enumitem}
\usepackage{pgfplots}
\usepgfplotslibrary{fillbetween}
\pgfplotsset{width=10cm,compat=1.9}
\usepackage[square,sort,comma,numbers]{natbib}
\allowdisplaybreaks

\usepackage{fancyhdr}
 
\numberwithin{equation}{section}

\makeatletter
\def\@setauthors{%
  \begingroup
  \def\thanks{\protect\thanks@warning}%
  \trivlist
  \centering\footnotesize \@topsep30\p@\relax
  \advance\@topsep by -\baselineskip
  \item\relax
  \author@andify\authors
  \def\\{\protect\linebreak}

  \normalsize\lowercase{\authors}%
  
	\ifx\@empty\contribs
  \else
    ,\penalty-3 \space \@setcontribs
    \@closetoccontribs
  \fi
  \endtrivlist
  \endgroup
}
\def\@settitle{\begin{center}
\LARGE\lowercase{\@title}
  \end{center}%
}
\makeatother

\definecolor{lightblue}{HTML}{2B77A4}
\definecolor{darkred}{HTML}{9E0D0D}
\hypersetup{
	colorlinks=true,
	linkcolor=darkred,
	urlcolor=darkred,
	citecolor=lightblue
}
\newtheorem{thm}{Theorem}[section]
\newtheorem{lma}[thm]{Lemma}
\newtheorem{cor}[thm]{Corollary}

\newtheorem{prop}[thm]{Proposition}

\newtheorem{ques}[thm]{Question}

\renewcommand{\epsilon}{\varepsilon}
\newcommand{\eps}{\varepsilon}

\newcommand{\rd}{\mathbb{R}^d}

\renewcommand{\geq}{\geqslant}
\renewcommand{\leq}{\leqslant}

\newcommand{\hd}{\dim_{\textup{H}}}
\newcommand{\pd}{\dim_{\textup{P}}}

\newcommand{\fs}{\dim^\theta_{\mathrm{F}}}
\newcommand{\fd}{\dim_{\mathrm{F}}}
\newcommand{\sd}{\dim_{\mathrm{S}}}

\title{The Fourier spectrum  and sumset type problems}

\author{Jonathan M. Fraser\\ \\
 University of St Andrews, Scotland\\
Email: jmf32@st-andrews.ac.uk}
\thanks{The  author was  financially supported by an \emph{EPSRC Standard Grant} (EP/R015104/1),  a  \emph{Leverhulme Trust Research Project Grant} (RPG-2019-034) and  an RSE Sabbatical Research Grant (70249).}

\begin{document}


\maketitle
\thispagestyle{empty}

\begin{abstract}
We introduce and study the \emph{Fourier spectrum} which is a continuously parametrised family  of dimensions living between the Fourier dimension and the Hausdorff dimension for both sets and measures. We establish some fundamental  theory and motivate the concept via several applications, especially to sumset type problems.  For example, we study dimensions of convolutions and sumsets, and solve the distance set problem for sets satisfying certain Fourier analytic conditions. 
\end{abstract}

\tableofcontents

\section{The Fourier spectrum: definition and basic properties}

The Hausdorff dimension (of a set or measure) is a fundamental geometric notion describing   fine scale structure.  The Fourier dimension, on the other hand, is an analytic notion which captures rather different features.  Both the Hausdorff and Fourier dimensions have  numerous applications in, for example, ergodic theory,  number theory, harmonic analysis and probability theory. The Fourier dimension of a set is bounded above by its  Hausdorff dimension and is much more sensitive to, for example,  arithmetic resonance and curvature. Indeed, the middle third Cantor set has Fourier dimension 0 because it possesses too much arithmetic structure, and a line segment embedded in the plane has Fourier dimension 0 because  it does not possess enough curvature. We note  that the Hausdorff dimension of both of these sets is strictly positive.  The line segment example also shows that the Fourier dimension is sensitive to the ambient space in a way that the Hausdorff dimension is not since a line segment in $\mathbb{R}$ has Fourier dimension 1.  

The purpose of this paper is to introduce, study, and motivate a continuously parametrised family of dimensions which  vary between the Fourier and Hausdorff dimensions. The hope is that the resulting `Fourier spectrum' will reveal more analytic and geometric information than the two notions considered in isolation and thus be amenable to applications in areas where both notions play a role.

We begin by defining the Fourier spectrum and deriving several fundamental properties including continuity (Theorems \ref{cty0} and \ref{ctyX}) and how   it depends on the ambient space  (Theorem \ref{embedding}) noting that the `endpoints' must behave rather differently.  We go on to put the work in a wider context, especially in relation to average Fourier dimensions and Strichartz type bounds (Theorems \ref{stric3} and \ref{stric1}).  During the above analysis we also derive (or estimate) the Fourier spectrum explicitly for several examples including Riesz products (Theorem \ref{riesz1} and Corollary \ref{riesz2}), certain self-similar measures (Corollaries \ref{selfsim1} and \ref{selfsim2}), and measures on various curves (Corollary \ref{hyperplane} and Theorem \ref{xp}). 

After establishing some fundamental theory, we move towards applications of the Fourier spectrum, especially concerning sumsets, convolutions, distance sets and certain random sets.  A rough heuristic which emerges is that when the Fourier spectrum is not (the restriction of) an affine function,   it provides more information than the Hausdorff and Fourier dimension on their own and this leads to new estimates in various contexts.  For example, the Sobolev dimension of a measure increases under convolution with itself if and only if the Fourier spectrum is not (the restriction of)  a linear function (Corollary \ref{conv2}) and when the Fourier spectrum of a measure is not an affine function, it provides better estimates for the Hausdorff dimension of the distance set  of its support than the Hausdorff and Fourier dimension provide on their own via Mattila integrals (Theorem \ref{maindistance}).  As a result we solve the distance set problem for sets satisfying certain Fourier analytic conditions.  A simple special case shows  that if $\mu$ is a finite Borel  measure on $\rd$ with $\int |\widehat \mu |^4 <\infty$, then the distance set of the support of $\mu$ has positive Lebesgue measure (Corollary \ref{cute}).   We also use the Fourier spectrum to give conditions ensuring a measure is `Sobolev improving' (Corollary \ref{sobolevimproving}),  to give estimates for the Hausdorff dimension of certain random constructions where the Fourier dimension alone provides only trivial estimates (Corollary \ref{randomsum}), and to provide a one line proof (and extension)  of a well-known connection between moment analysis and Fourier dimension in random settings (Lemma \ref{moment}).

The idea to introduce a continuum of dimensions in-between a given pair of `fractal  dimensions' is  part of a growing programme sometimes referred to as `dimension interpolation'.  Previous examples include the \emph{Assouad spectrum} which lives in-between the (upper) box-counting and Assouad dimensions \cite{assouad} and the \emph{intermediate dimensions} \cite{intermediate} which live in-between the Hausdorff and box-counting dimensions.  The Fourier spectrum is of a rather different flavour since the aforementioned notions are defined via coverings.  Despite their  recent inception, the Assouad spectrum and intermediate dimensions are proving useful tools in a growing range of (often unexpected) areas, for example, in quasi-conformal mapping theory \cite{tyson} and in analysis of spherical maximal functions \cite{joris}. We believe this will also be the case for the Fourier spectrum.

\subsection{Background: Energy,   Fourier transforms, and dimension}

Throughout the  paper, we write $A \lesssim B$ to mean there exists a constant $c >0$ such that $A \leq cB$.  The implicit constants $c$ are suppressed to improve exposition.  If we wish to emphasise that these constants depend on another parameter $\lambda$, then we will write $A \lesssim_\lambda B$.  We also write $A \gtrsim B$ if $B \lesssim A$ and $A \approx B$ if $A \lesssim B$ and $A \gtrsim B$.

Let $\mu$ be a finite Borel measure on $\mathbb{R}^d$ with support denoted by $\textup{spt}(\mu)$.  
For $s \geq 0$, the $s$-energy of $\mu$ is  given by 
\[
\mathcal{I}_s(\mu) = \int \int \frac{d \mu(x) \, d \mu(y)}{|x-y|^s} 
\]
and can be used to estimate the Hausdorff dimension of $\mu$ and its support (the so-called `potential theoretic method').  Indeed, if $s \geq 0$ is such that $\mathcal{I}_s(\mu) <\infty$, then
\[
\hd \textup{spt}(\mu) \geq \hd \mu \geq s
\]
where $\hd$ denotes Hausdorff dimension. In fact, this is a precise characterisation of Hausdorff dimension since for all Borel sets $X$ and $s < \hd X$, there exists a finite Borel measure $\mu$ on $X$ such that $\mathcal{I}_s(\mu) <\infty$.   See \cite{falconer, mattila} for more on Hausdorff dimension, energy and the potential theoretic method.

There is an elegant connection between energy (and thus Hausdorff dimension) and the Fourier transform.   The Fourier transform of $\mu$ is the function $\widehat \mu : \rd \to \mathbb{C}$ given  by
\[
\widehat \mu (z) = \int e^{-2\pi i z \cdot x} \, d \mu(x).
\]
In fact, for $0<s<d$ 
\[
\mathcal{I}_s(\mu) \approx_{s,d} \int_{\mathbb{R}^d} |\widehat \mu(z)  |^{2} |z|^{s-d} \, dz,
\]
see \cite[Theorem 3.1]{mattila}. Therefore, if $|\widehat \mu(z)  |\lesssim |z|^{-s/2}$ for some $s \in (0,d)$, then $\mathcal{I}_s(\mu) <\infty$ and $\hd \mu \geq s$.  This motivates the \emph{Fourier dimension}, defined by
\[
\fd \mu = \sup\{ s \geq 0 : |\widehat \mu(z)  |\lesssim |z|^{-s/2}\}
\]
for measures and
\[
\fd X =   \sup\big\{\min\{ \fd \mu,  d\}  : \textup{spt}(\mu) \subseteq X  \big\}
\]
for  sets $X \subseteq \rd$. Here the supremum is taken over finite Borel measures $\mu$ supported by   $X$.  See \cite{modfourier} for some interesting alternative formulations of the Fourier dimension, including the modified Fourier dimension and the compact Fourier dimension.  For non-empty sets   $X \subseteq \rd$ 
\[
0 \leq \fd X \leq \hd X \leq d
\]
and $X$ is a \emph{Salem set} if and only if $\fd X = \hd X$.  See \cite{mattila} for more on the Fourier dimension and \cite{grafakos} for background on Fourier analysis more generally. There are many random constructions giving rise to Salem sets but non-trivial deterministic examples are much harder to come by and in general it is rather easy for a set to fail to be Salem.  For example, a line segment in $\mathbb{R}$ is Salem but in $\mathbb{R}^2$ it is not.  Further, the middle third Cantor set, many self-affine sets, the cone in $\mathbb{R}^3$, and the graph of Brownian motion all fail to be Salem sets.    In this article we are interested in   sets which are \emph{not} Salem and wish to explore the difference between $\fd X$ and $\hd X$ in a novel and  meaningful way.

\subsection{The Fourier spectrum}

We  exploit the connection between Fourier dimension and Hausdorff dimension via energy to define  a continuum of `dimensions' lying in-between the Fourier and Hausdorff dimensions. Let $\mu$ be a finite Borel measure on $\mathbb{R}^d$.  For $\theta \in [0,1]$ and $s \geq 0$, we define energies
\[
\mathcal{J}_{s,\theta}(\mu) = \left( \int_{\mathbb{R}^d} |\widehat \mu(z)  |^{2/\theta} |z|^{s/\theta-d} \, dz \right)^{\theta},
\]
where we adopt the convention that
\[
\mathcal{J}_{s,0}(\mu) = \sup_{z \in \mathbb{R}^d} |\widehat \mu(z)  |^{2} |z|^{s}.
\]
We then define the \emph{Fourier spectrum} of $\mu$ at $\theta$ by 
\[
\fs \mu = \sup\{s \geq 0 : \mathcal{J}_{s,\theta}(\mu) < \infty\}
\]
where we write $\sup \o = 0$.  Note that 
\[
\mathcal{J}_{s,1}(\mu) = \int_{\mathbb{R}^d} |\widehat \mu(z)  |^{2} |z|^{s-d} \, dz 
\]
is  the familiar \emph{Sobolev energy} and, therefore, $\dim^1_{\mathrm{F}} \mu = \sd \mu$ where $\sd \mu$ is the \emph{Sobolev dimension} of $\mu$, see \cite[Section 5.2]{mattila}.  Moreover,  $\dim^0_{\mathrm{F}} \mu = \dim_{\mathrm{F}} \mu$ is the Fourier dimension of $\mu$. 

For   sets  $X \subseteq \mathbb{R}^d$, we define  the \emph{Fourier spectrum}   of $X$ at $\theta$ by 
\[
\fs X =   \sup\big\{ \min\{\fs \mu , d\} :  \textup{spt}(\mu) \subseteq X  \big\}.
\]
Here the supremum is again taken over finite Borel measures $\mu$ supported by  $X$.  One immediately sees that $\dim^0_{\mathrm{F}} X = \dim_{\mathrm{F}} X$ is the Fourier dimension of $X$. Moreover, using  $\mathcal{J}_{s,1}(\mu)  \approx  \mathcal{I}_s(\mu)$ for $0<s<d$ where  $\mathcal{I}_s(\mu)$ is the standard energy (see \cite[Theorem 3.1]{mattila}),  
 $\dim^1_{\mathrm{F}} X = \dim_{\mathrm{H}} X$ returns  the Hausdorff dimension for Borel sets $X$. The quantity $\min\{\sd \mu , d\} = \min\{ \dim^1_{\mathrm{F}} \mu, d\}$ is the \emph{energy dimension} or \emph{lower correlation dimension} of $\mu$. It is easy to see (e.g. Theorem \ref{concave} below) that
\[
\dim_{\mathrm{F}} \mu  \leq \fs \mu \leq \dim_{\mathrm{S}} \mu 
\]
and
\[
\dim_{\mathrm{F}} X  \leq \fs X \leq \dim_{\mathrm{H}} X
\]
for all $\theta \in (0,1)$.

The Fourier spectrum can be defined in terms of $L^p$ spaces in a convenient way.  Define a measure $m_d$ on $\mathbb{R}^d$  by  $d m_d =\min\{|z|^{-d}, 1\} \, dz$ and a family of functions  $f^s_\mu : \mathbb{R}^d \to \mathbb{R}$ by $f_\mu^s(z) =|\widehat \mu(z)|^{2} |z|^{s}$.  Then 
\[
\mathcal{J}_{s,\theta}(\mu) \approx \| f^s_\mu \|_{L^{1/\theta}(m_d)}
\]
for $s>0$ and so
\[
\fs \mu = \sup\left\{ s: f^s_\mu \in L^{1/\theta}(m_d) \right\}.
\]

It could be interesting to consider a \emph{modified Fourier spectrum}, following the modified Fourier dimension defined in \cite{modfourier}, but we do not pursue this here.  We will be more focused on the Fourier spectrum of measures and many of our results for sets would also hold for the modified or compact variants considered by \cite{modfourier}. We leave details to the reader.

Finally, we hope that our use of `Fourier spectrum' does not cause confusion with other uses of the phrase in the literature, for example \cite{bourgain}.  In a previous draft of the paper we opted for the `Fourier dimension spectrum' to avoid this issue, but in practice this name was too cumbersome and was always shortened.  Further, we prefer to align the \emph{Fourier} spectrum philosophically with the \emph{Assouad} spectrum and hence the general programme of dimension interpolation.

\begin{figure}[H]
							\includegraphics[width=\textwidth]{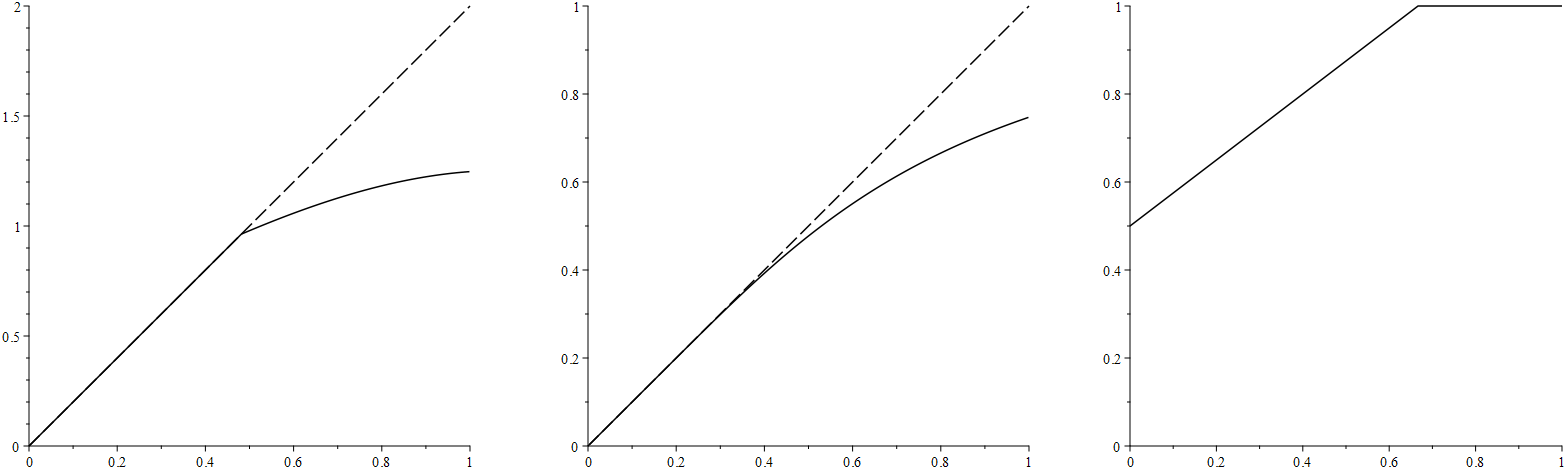} \\ \vspace{3mm}

			\caption{\emph{Three examples.}  Left: the Fourier  spectrum of a measure with positive Fourier dimension in the plane  embedded into $\mathbb{R}^3$, see Theorem  \ref{embedding}.  Centre: the Fourier spectrum of a Riesz product from Corollary \ref{riesz2} with $a=0.8, \lambda=3$.  Right: the Fourier spectrum of Lebesgue measure lifted onto the graph of $x \mapsto x^4$, see Theorem \ref{xp}.}
			\end{figure}

\subsection{Analytic  properties: continuity, concavity} \label{basicsection}

The first task is to examine fundamental properties of the function $\theta \mapsto \fs \mu$.  The results we prove in this section will be used throughout the paper as we develop the theory towards more sophisticated applications.

\begin{thm} \label{concave}
Let $\mu$ be a finite Borel measure and $X$ a non-empty set.  Then $\fs \mu$  is a non-decreasing concave function of $\theta \in [0,1]$.  In particular, $\fs \mu$ is continuous on $(0,1]$.  Further, $\fs X$ is non-decreasing on $[0,1]$ and continuous on $(0,1]$ but may not be concave.  
\end{thm}

\begin{proof}
The claims for $X$ are clear once the claims for $\mu$ are established.  We first prove concavity.  Fix $0 \leq \theta_0 < \theta_1 \leq 1$ and let $\theta \in (\theta_0, \theta_1)$.  Let $s_0<\dim_\mathrm{F}^{\theta_0} \mu$,  $s_1<\dim_\mathrm{F}^{\theta_1} \mu$ and
\[
s= s_0 \frac{\theta_1-\theta}{\theta_1-\theta_0} + s_1\frac{\theta-\theta_0}{\theta_1-\theta_0}
\]
and assume without loss of generality that $s>0$.  Define $m_d$ by $dm_d =\min\{|z|^{-d}, 1\} \, dz$.  Then
\begin{align*}
\mathcal{J}_{s,\theta}(\mu)^{1/\theta} &\approx \int_{\mathbb{R}^d} |\widehat \mu(z)  |^{2/\theta} |z|^{s/\theta}   \, dm_d(z) \\
&  =  \int_{\mathbb{R}^d}  \left( |\widehat \mu(z)  |^{2/\theta_0} |z|^{s_0/\theta_0}\right)^{  \frac{\theta_0(\theta_1-\theta)}{\theta(\theta_1-\theta_0)} } \, \left( |\widehat \mu(z)  |^{2/\theta_1} |z|^{s_1/\theta_1}\right)^{ \frac{\theta_1(\theta-\theta_0)}{\theta(\theta_1-\theta_0)} } \, dm_d(z) \\
&  \leq  \left( \int_{\mathbb{R}^d} |\widehat \mu(z)  |^{2/\theta_0} |z|^{s_0/\theta_0}  \, dm_d(z) \right)^{\frac{\theta_0(\theta_1-\theta)}{\theta(\theta_1-\theta_0)}}  \ \left( \int_{\mathbb{R}^d}  |\widehat \mu(z)  |^{2/\theta_1} |z|^{s_1/\theta_1}  \, dm_d(z) \right)^{\frac{\theta_1(\theta-\theta_0)}{\theta(\theta_1-\theta_0)}} \\
&<\infty
\end{align*}
by H\"older's inequality and choice of $s_0$ and $s_1$.  This establishes  $\fs \mu \geq s$, proving concavity of $\fs \mu$.  

Next we prove that $\fs \mu$ is non-decreasing. Fix $0 \leq \theta_0 < \theta_1 \leq 1$.  Let $\eps>0$ and define $m_d^\eps$ by  $dm_d^\eps =c \min\{|z|^{-(d+\varepsilon)}, 1\} \, dz$ where $c$ is chosen such that $m_d^\eps$ is a probability measure.  Then, using Jensen's inequality,
\begin{align*}
\mathcal{J}_{s,\theta_1}(\mu)^{1/\theta_0} &\approx \left(\int_{\mathbb{R}^d} |\widehat \mu(z)  |^{2/\theta_1} |z|^{s/\theta_1} |z|^{\eps}  \, dm_d^\eps(z) \right)^{\theta_1/\theta_0} \\
&\leq \int_{\mathbb{R}^d} |\widehat \mu(z)  |^{2/\theta_0} |z|^{s/\theta_0} |z|^{\eps\theta_1/\theta_0}  \, dm_d^\eps(z)  \\
&= \int_{\mathbb{R}^d} |\widehat \mu(z)  |^{2/\theta_0} |z|^{(s +\eps( \theta_1-\theta_0))/\theta_0}   |z|^{\eps} \, dm_d^\eps(z)  \\
&\approx \mathcal{J}_{s+\eps( \theta_1-\theta_0),\theta_0}(\mu)^{1/\theta_0} \\
&<\infty
\end{align*}
provided $s< \dim_\mathrm{F}^{\theta_0} \mu -\eps(\theta_1-\theta_0)$.  Letting $\eps$ tend to 0 proves $\dim_\mathrm{F}^{\theta_1} \mu \geq \dim_\mathrm{F}^{\theta_0} \mu $ as required. 
Finally, a non-decreasing concave function on $[0,1]$ is immediately seen to be continuous on $(0,1]$.
\end{proof}

Later we will show that the Fourier spectrum is not necessarily continuous at $\theta = 0$, see Proposition \ref{discont}.  However, continuity can be established across the whole range $[0,1]$ assuming only very mild conditions. In order to prove continuity of the Fourier spectrum at $\theta=0$ (and thus over the full range $[0,1]$)  we need to assume H\"older continuity of the Fourier transform.  First we show that this holds assuming only a mild decay condition on the measure.  For compactly supported measures the Fourier transform is Lipschitz (see \cite[(3.19)]{mattila})  but this is true for many non-compactly supported measures too.

\begin{lma} \label{holder}
Let $\mu$ be a finite Borel measure on $\mathbb{R}^d$     such that $|z|^\alpha \in L^1(\mu)$ for some $\alpha \in (0,1]$.  Then $\widehat \mu$ is $\alpha$-H\"older. 
\end{lma}

\begin{proof}
For $x, y \in \rd$ with $|x-y| \geq 1$, it is immediate that $|\widehat \mu(x) - \widehat\mu(y) | \leq 2 \leq 2|x-y|$ and so we may assume $|x-y| \leq 1$.  Then,
\begin{align*}
|\widehat \mu(x) - \widehat\mu(y) |  \leq  \int \left\lvert 1-e^{-2\pi i (x-y) \cdot z } \right\rvert \, d \mu(z)&  \lesssim \int \left\lvert 1-e^{-2\pi i (x-y) \cdot z } \right\rvert^\alpha \, d \mu(z) \\
&\lesssim  \int |x-y|^\alpha|z|^\alpha  \, d \mu(z) \\
&\lesssim |x-y|^\alpha
\end{align*}
as required.
\end{proof}

\begin{thm} \label{cty0}
Let $\mu$ be a finite Borel measure on $\rd$  such that $\widehat \mu$ is $\alpha$-H\"older.  Then
\[
\fs \mu \leq \fd \mu +  d\left(1+\frac{\fd \mu}{2\alpha} \right)\theta.
\]
In particular, $\fs \mu$  is Lipschitz continuous at $\theta=0$ and therefore Lipschitz continuous on $[0,1]$.
\end{thm}

\begin{proof}
 Let $t>\fd \mu$ which guarantees the existence of a sequence $z_k \in \mathbb{R}^d$ with $|z_k| \geq 1$ and $|z_k| \to \infty$  such that for all $k$
\[
|\widehat \mu (z_k)| \geq |z_k|^{-t/2}.
\]
Since $\mu$ is   $\alpha$-H\"older,  there exists $c=c(\mu) \in (0,1)$ such that
\[
|\widehat \mu (z)| \geq |z_k|^{-t/2}/2
\]
for all $z \in B(z_k,c|z_k|^{-t/(2\alpha)})$.  By passing to a  subsequence if necessary we may assume that the balls $B(z_k,c|z_k|^{-t/(2\alpha)})$ are pairwise disjoint. Therefore
\begin{align*}
\mathcal{J}_{s,\theta}(\mu)^{1/\theta} &\geq \sum_k \int_{B(z_k,c|z_k|^{-t/(2\alpha)})} |\widehat \mu(z)  |^{2/\theta} |z|^{s/\theta-d} \, dz  \\
&\gtrsim \sum_k |z_k|^{-td/(2\alpha)} |z_k|^{-t/\theta} |z_k|^{s/\theta - d} = \infty
\end{align*}
whenever $s\geq  d \theta + t+td\theta/(2\alpha)$.  This proves
\[
\fs \mu \leq \fd \mu +  d\left(1+\frac{\fd \mu}{2\alpha} \right)\theta 
\]
as required.  The final conclusion  that $\fs \mu$  is Lipschitz continuous   on $[0,1]$ follows immediately from  Lipschitz continuity  at $\theta=0$ together with the fact  that  $\fs \mu$ is non-decreasing and concave, see Theorem \ref{concave}.
\end{proof}

Later we provide an example showing that the   bounds from Theorem \ref{cty0} are sharp in the case $\fd \mu = 0$ and $\alpha = 1$, see Corollary \ref{firstsharp}.  

\begin{ques}
Are the bounds from Theorem \ref{cty0} sharp?  Are they sharp when $\mu$ is compactly supported and $\fd \mu >0$?
\end{ques}

 One benefit of Theorem \ref{cty0} is the demonstration that positive Fourier dimension of a measure can always be observed by averaging (in the case when $\widehat \mu$ is H\"older).  That is, provided $\widehat \mu$ is H\"older, $\fd \mu >0$ if and only if $\fs \mu>d \theta$ for some $\theta >0$.  This is potentially useful since positive Fourier dimension requires uniform estimates on $|\widehat \mu|$ which are \emph{a priori} harder to obtain than  estimates on the `averages'  $\mathcal{J}_{s,\theta}(\mu)$.

Using  Theorem \ref{cty0} we immediately get continuity of the Fourier spectrum for compact sets.   However, using a trick inspired by \cite[Lemma 1]{modfourier} we can upgrade this to  continuity of the Fourier spectrum for \emph{all} sets.

\begin{thm} \label{ctyX}
If   $ X\subseteq \rd $ is a non-empty   set, then
\[
\fs X \leq \fd X +  d\left(1+\frac{\fd X}{2} \right)\theta 
\]
for all $\theta \in [0,1]$. In particular,   $\fs X$ is Lipschitz continuous on the whole range $[0,1]$ with $\fd^0 X = \fd X$ and,  if $X$ is Borel, $\fd^1 X = \hd X$.  
\end{thm}

\begin{proof}
Let $\mu$ be a finite Borel measure supported by $X$.  Let $f: \rd \to [0,\infty)$ be a smooth function with compact support such that the  Borel  measure $\nu$ defined by $d \nu = f \, d \mu$ satisfies $\nu(X)>0$.  Then $\nu$ is  supported on a compact subset of $X$.  We claim that
\[
\fs \nu \geq \fs \mu
\]
for all $\theta \in [0,1]$. Together with Theorem \ref{cty0} and Lemma \ref{holder}, this  claim proves the result.  Since $f$ is smooth and has compact support, it holds that for all integers $n \geq 1$ $|\widehat f (t)| \lesssim_n |t|^{-n}$ for $|t| \geq 1$. In particular, $\widehat f$ and $f$ are  $L^1$ functions. Therefore, by the Fourier inversion formula,
\begin{equation} \label{inversionuse}
\widehat \nu (z) = \int_{\rd} \widehat \mu(z-t) \widehat f(t) \, dt.
\end{equation}
The claim for $\theta=0$ is \cite[Lemma 1]{modfourier}  and so we may assume $\theta \in (0,1]$.  Using \eqref{inversionuse}
\begin{align*}
\mathcal{J}_{s,\theta}(\nu)^{1/\theta} &= \int_{\rd} \left\lvert \int_{\rd} \widehat \mu(z-t) \widehat f(t) \, dt \right\rvert^{2/\theta} |z|^{s/\theta-d} \, dz\\
& \lesssim  \int_{\rd}\int_{\rd}  \left\lvert \widehat \mu(z-t)\right\rvert^{2/\theta} |\widehat f(t) | \, dt  \, |z|^{s/\theta-d} \, dz \qquad \text{(by Jensen's inequality)}\\
& =  \int_{\rd}|\widehat f(t) | \int_{\rd}  \left\lvert \widehat \mu(z-t)\right\rvert^{2/\theta}|z|^{s/\theta-d} \, dz \, dt   \qquad \text{(by Fubini's  theorem)}\\
& =  \int_{\rd}|\widehat f(t) | \int_{\rd}  \left\lvert \widehat \mu(z)\right\rvert^{2/\theta}|z+t|^{s/\theta-d} \, dz \, dt    \\
& =  \int_{\rd}|\widehat f(t) | \int_{\rd \setminus B(0,5 |t|)}  \left\lvert \widehat \mu(z)\right\rvert^{2/\theta}|z+t|^{s/\theta-d} \, dz \, dt  \\
&\, \hspace{4cm} +    \int_{\rd}|\widehat f(t) | \int_{B(0,5 |t|)}  \left\lvert \widehat \mu(z)\right\rvert^{2/\theta}|z+t|^{s/\theta-d} \, dz \, dt    \\   
&\lesssim  \int_{\rd}|\widehat f(t) | \, dt  \int_{\rd}  \left\lvert \widehat \mu(z)\right\rvert^{2/\theta}|z|^{s/\theta-d} \, dz  +  \int_{\rd}|\widehat f(t) | |t|^d |t|^{s/\theta-d} \, dt     \\
&\lesssim \mathcal{J}_{s,\theta}(\mu)^{1/\theta} + 1
\end{align*}
using the rapid decay of $|\widehat f(t)|$.  This proves the claim and the theorem.  
\end{proof}

It is useful to keep the following simple bounds in mind. These are immediate from Theorems \ref{concave} and \ref{cty0}.

\begin{cor} \label{bounds}
Let $\mu$ be a compactly supported  finite Borel measure on $\mathbb{R}^d$.  Then
\[
 \fd \mu + \theta\left(\sd \mu - \fd \mu\right)  \leq \fs \mu \leq  \min \left\{ \sd \mu, \, \fd \mu +  d\left(1+\frac{\fd \mu}{2} \right)\theta \right\}.
\]
\end{cor}

In certain extremal situations the Fourier spectrum is determined by the Fourier and Sobolev dimensions.

\begin{cor} \label{collapse}
If  $\mu$ is a finite Borel measure on $\rd$  such that $\widehat \mu$ is $\alpha$-H\"older and 
\[
\sd \mu = \left(1+\frac{d}{2\alpha} \right) \fd \mu + d,
\]
then
\[
\fs \mu  =  \fd \mu +  d\left(1+\frac{\fd \mu}{2\alpha} \right)\theta 
\]
for all $\theta \in [0,1]$.
\end{cor}

Another simple consequence of Theorem \ref{cty0} is that the Sobolev dimension can be controlled by the Fourier dimension. We are unaware if the following estimates were known previously. 

\begin{cor} \label{sobolev}
Let $\mu$ be a finite Borel measure on $\rd$  such that $\widehat \mu$ is $\alpha$-H\"older.   Then
\[
\sd \mu \leq \left(1+\frac{d}{2\alpha} \right) \fd \mu + d.
\]
In particular, if 
\[
\fd \mu \leq  \frac{d}{1+d/(2\alpha)}
\]
then $\sd \mu \leq  2d$, and   if $\fd \mu =0$, then  $\sd \mu \leq  d$.  These are  relevant thresholds for Sobolev dimension because if $\sd \mu > 2d$, then $\mu$ is a continuous function and if $\sd \mu >d$, then $\mu \in L^2(\mathbb{R}^d)$, see  \cite[Theorem 5.4]{mattila}.
\end{cor}

\section{Dependence on ambient space}

An elementary but striking observation on the Fourier dimension is that it depends on the ambient space in a way that Hausdorff dimension, for example, does not.  Consider integers $1 \leq k<d$ and let $f^{k,d} : \mathbb{R}^k \to \mathbb{R}^d$ be an isometric embedding defined by identifying $\mathbb{R}^k$ with a $k$-dimensional affine subset of $\mathbb{R}^d$.  Then it is well-known and easily seen that
\[
\fd f_\#^{k,d} \mu = 0
\]
for all finite Borel measures $\mu$ on $\mathbb{R}^k$ but, provided $\sd \mu \leq k$,
\[
\sd f_\#^{k,d} \mu = \sd \mu.
\]
Here $f_\#^{k,d} \mu $ is the pushforward of $\mu$ under $f^{k,d}$. The conclusion for Sobolev dimension follows by observing that the standard energy $\mathcal{I}_s(\mu)$ does not depend on the ambient space. Moreover,
\[
\fd f^{k,d}(X) = 0
\]
and
\[
\hd f^{k,d}(X) = \hd X
\]
for all $X \subseteq \mathbb{R}^k$. Since the Fourier spectrum interpolates between the Fourier and Sobolev/Hausdorff dimensions, it is natural to ask what happens to the Fourier spectrum under the  embeddings $f^{k,d}$. In particular, the answer must encapsulate both of the rather distinct behaviours seen above.

\begin{thm} \label{embedding}
Let $1 \leq k<d$ be integers,  $\mu$ be a  finite Borel measure on $\mathbb{R}^k$ and $X \subseteq \mathbb{R}^k$ be a non-empty set.  Then
\[
\fs  f_\#^{k,d} \mu = \min\{\theta k, \fs \mu\}
\]
and
\[
\fs  f ^{k,d} (X) = \min\{\theta k, \fs X\}
\]
for all $\theta \in [0,1]$.
\end{thm}

\begin{proof}
The claim for sets follows immediately from the claim for measures and so we just prove the result for $\mu$.  Fix $\theta \in (0,1)$, write $V = f^{k,d}(\mathbb{R}^k)$ and  $\pi$ for orthogonal projection from $\mathbb{R}^d$ onto $V$ identified with $\mathbb{R}^k$.  We may assume without loss of generality that $V$ is a subspace of $\rd$.  We begin with the upper bound. Since $\widehat \mu (0) = \mu(\rd)>0$ and $\widehat \mu $ is continuous, there exists $\eps=\eps(\mu)$ such that $|\widehat \mu (z)| \geq \mu(\rd)/2>0$ for all $z \in B_V(0,\eps)$, where $B_V$ denotes   open balls in $V$.  Write $0 \neq z \in \mathbb{R}^d$ in spherical coordinates $z=r v$ where $r>0$ and $v \in S^{d-1}$ with $\sigma_{d-1}$ the surface measure on $S^{d-1}$. Then
\begin{align*}
\mathcal{J}_{s,\theta}( f_\#^{k,d} \mu)^{1/\theta} &=\int_{0}^\infty  \int_{S^{d-1}} | \widehat \mu(r \pi(v))  |^{2/\theta} r^{s/\theta-d} \, r^{d-1} d\sigma_{d-1}(v) dr \\
&  \geq \int_{0}^\infty r^{s/\theta-1} \int_{\substack{v \in S^{d-1}: \\ r \pi(v) \in B_V(0,\eps)}} |\widehat \mu(r \pi(v))  |^{2/\theta}  \, d\sigma_{d-1}(v) dr \\
&  \gtrsim_\theta  \int_{0}^\infty r^{s/\theta-1}\sigma_{d-1}\left(v \in S^{d-1}:   \pi(v) \in B_V(0,\eps/r) \right)  dr \\
&  \gtrsim  \int_{0}^\infty r^{s/\theta-1} \left( \eps/r\right)^k  dr \\
&=\infty
\end{align*}
whenever $s \geq k\theta$, proving $\fs f_\#^{k,d} \mu \leq k \theta$.

  For the remaining upper bound, we may assume $\theta k >\fs \mu$ and   let $\theta k> s >\fs \mu$ and fix  $1< c <2$ and $\max\{c/2,1/c\} < r <1$. In what follows the implicit constants may depend on $r$ and $c$ (and other fixed parameters as usual).  Then,  writing $z=(x,y)$ for $x \in V^\perp$ and $y \in V$,
\begin{align*}
\int_{\mathbb{R}^d} |\widehat{ f_\#^{k,d} \mu}(z)  |^{2/\theta} |z|^{s/\theta-d} \, dz & \gtrsim  \int_{x \in V^\perp } \int_{\substack{y \in V : \\ |x| \leq |y| \leq 2 |x|}} |\widehat \mu(y)  |^{2/\theta} |y|^{s/\theta-k}  |x|^{k-d}  \, dy dx \\
& \geq \sum_{n=1}^\infty  \int_{\substack{x \in V^\perp: \\ r c^n \leq |x| \leq c^n} } \int_{\substack{y \in V : \\ c^n \leq |y| \leq c^{n+1}}} |\widehat \mu(y)  |^{2/\theta} |y|^{s/\theta-k}  |x|^{k-d}  \, dy dx \\
& \gtrsim \sum_{n=1}^\infty  c^{n(k-d)}\int_{\substack{x \in V^\perp: \\ r c^n \leq |x| \leq c^n} }  \,   dx \int_{\substack{y \in V : \\ c^n \leq |y| \leq c^{n+1}}} |\widehat \mu(y)  |^{2/\theta} |y|^{s/\theta-k}    \, dy  \\
& \gtrsim \sum_{n=1}^\infty  \int_{\substack{y \in V : \\ c^n \leq |y| \leq c^{n+1}}} |\widehat \mu(y)  |^{2/\theta} |y|^{s/\theta-k}    \, dy  \\
& =  \int_{\substack{y \in V : \\ c \leq |y| }} |\widehat \mu(y)  |^{2/\theta} |y|^{s/\theta-k}    \, dy  \\
&=\infty
\end{align*}
proving $\fs f_\#^{k,d} \mu \leq \fs \mu$.

We turn our attention to the lower bound.  Let $s <\min\{\theta k , \fs \mu\}$ and $\eps \in (0,1)$ with $s+\eps\theta <\min\{\theta k , \fs \mu\}$.  Then, again writing $z=(x,y)$ for $x \in V^\perp$ and $y \in V$,
\begin{align*}
\int_{\mathbb{R}^d} |\widehat{ f_\#^{k,d} \mu}(z)  |^{2/\theta} |z|^{s/\theta-d} \, dz & =  \int_{\mathbb{R}^d} |\widehat \mu (\pi(z))  |^{2/\theta} |z|^{(s+\eps \theta)/\theta-k} |z|^{k-d-\eps} \, dz \\
& \leq   \int_{x \in V^\perp } \int_{y \in V } |\widehat \mu (y)  |^{2/\theta}  |y|^{(s+\eps \theta)/\theta-k} |x|^{k-d-\eps}  \, dy \, dx \\
& \leq   \int_{y \in V } |\widehat \mu (y)  |^{2/\theta}  |y|^{(s+\eps \theta)/\theta-k}  \, dy   \int_{x \in V^\perp }|x|^{k-d-\eps}  \, dx \\
&<\infty
\end{align*}
proving $\fs  f_\#^{k,d} \mu \geq  \min\{\theta k, \fs \mu\}$, as required.
\end{proof}

One can see from Theorem \ref{embedding} together with Theorem \ref{cty0} that the Fourier spectrum of a measure is preserved upon embedding in a Euclidean space with strictly larger dimension if and only if the Fourier dimension of the measure is zero to begin with.

Using Theorem \ref{embedding} we get our first  examples where the Fourier spectrum can be derived explicitly.

\begin{cor} \label{hyperplane}
Let $X$ be an isometric embedding of  $[0,1]^k$ in  $\mathbb{R}^d$ for integers $1 \leq k < d$ and let  $\mu$ be the restriction of $k$-dimensional Hausdorff measure to $X$.  Then
\[
\fs X = \fs \mu = k \theta
\]
for all $\theta \in [0,1]$. 
\end{cor}

\section{Fourier coefficients, energy, and Riesz products}

There is a convenient representation for the energy in terms of the Fourier coefficients of a measure, see \cite[Theorem 3.21]{mattila}.  Indeed, for $\mu$ a  finite Borel measure on $\rd$ with support contained in $[0,1]^d$
\begin{equation} \label{fouriercoeff}
\mathcal{J}_{s,1}(\mu) \approx 1 +  \sum_{z\in \mathbb{Z}^d \setminus\{0\}} |\widehat{\mu}(z)|^2 |z|^{s-d}  
\end{equation}
for $0<s<d$. 
Using the convolution formula, \eqref{fouriercoeff}  gives information about $\mathcal{J}_{s,\theta}(\mu)$ for $0<s<d \theta$ with $\theta$ the reciprocal of an integer. However, we need to sum over the  finer grid $\theta\mathbb{Z}^d$ because the convolution $\mu^{*1/\theta}$ is no-longer supported on $[0,1]^d$ (but on $[0,1/\theta]^d$) and so we need to rescale in order to apply \eqref{fouriercoeff}.

\begin{prop} \label{coeffs}
Let $\mu$ be a  finite Borel measure on $\rd$ with support contained in $[0,1]^d$. For $0<s<d \theta$ and $\theta =1/k$ for $k \in \mathbb{N}$ 
\[
\mathcal{J}_{s,\theta}(\mu)   \approx   \left(1 +  \sum_{z\in \theta\mathbb{Z}^d \setminus\{0\}} |\widehat{\mu}(z)|^{2/\theta} |z|^{s/\theta-d}\right)^\theta .
\]
\end{prop}
\begin{proof}
 Writing  $\theta \mu^{*1/\theta}$ for the pushforward of  $\mu^{*1/\theta}$ under the dilation $x \mapsto \theta x$, we get that $\theta \mu^{*1/\theta}$  is  supported on $[0,1]^d$ and that
\begin{equation} \label{temp}
\widehat{\theta \mu^{*1/\theta}}(z) = \widehat{ \mu^{*1/\theta}}(\theta z)
\end{equation}
for $z \in \mathbb{R}^d$.  Then
\begin{align*}
\mathcal{J}_{s,\theta}(\mu)^{1/\theta} & = \int_{\rd} |\widehat \mu (z)|^{2/\theta} |z|^{s/\theta - d} \, dz \\
& = \int_{\rd} |\widehat{ \mu^{*1/\theta}} (z)|^{2} |z|^{s/\theta - d} \, dz \qquad \text{(convolution formula)}\\
& \approx_\theta  \int_{\rd} | \widehat{\theta \mu^{*1/\theta}}(z/\theta)|^{2} |z/\theta|^{s/\theta - d} \, dz \qquad \text{(by \eqref{temp})} \\
& \approx_\theta   \int_{\rd} | \widehat{\theta \mu^{*1/\theta}}(z)|^{2} |z|^{s/\theta - d} \, dz\\
& = \mathcal{J}_{s/\theta,1}(\theta \mu^{*1/\theta}) \\
 &\approx 1 +  \sum_{z\in \mathbb{Z}^d \setminus\{0\}} |\widehat{\theta \mu^{*1/\theta}}(z)|^2 |z|^{s/\theta-d}  \qquad \text{(by \eqref{fouriercoeff})}  \\
 &\approx_\theta 1 +  \sum_{z\in \mathbb{Z}^d \setminus\{0\}} |\widehat{ \mu^{*1/\theta}}(\theta z)|^2 |\theta z|^{s/\theta-d}  \qquad \text{(by \eqref{temp})}  \\
& = 1 +  \sum_{z\in \theta\mathbb{Z}^d \setminus\{0\}} |\widehat{\mu}(z)|^{2/\theta} |z|^{s/\theta-d}  \qquad \text{(convolution formula)}
\end{align*}
as required.
\end{proof}

It would be useful to relax the assumptions in Proposition \ref{coeffs} but we do not attempt this here.
\begin{ques}
In  Proposition \ref{coeffs}, can the assumption that $\theta$ is the reciprocal of an integer be removed?  Can the assumption that $0<s<d\theta$ be weakened?
\end{ques}

Being able to estimate the energy in terms of the Fourier coefficients often simplifies calculations.  To demonstrate this we give an easy example of a measure showing the upper bound from Theorem \ref{cty0} is sharp in the case $\fd \mu = 0$ and $\alpha = 1$.

\begin{cor} \label{firstsharp}
Define $f:[0,1] \to \mathbb{R}$   by
\[
f(x) = 2+ \sum_{n=1}^\infty n^{-2} \sin(2\pi2^{n}x)
\]
and $f_d: [0,1]^d \to \mathbb{R}$ by $f_d(x_1, \dots, x_d) = f(x_1)  \cdots  f(x_d)$.  
Then $f$ and $f_d$ are non-negative and we may define a measure $\mu$ supported on $[0,1]^d$ by  $d\mu = f_d dx$. Then
\[
\fs \mu = \theta d
\]
for all $\theta \in [0,1]$.
\end{cor}

\begin{proof}
 For integers $n \geq 1$
 \[
|\widehat f(2^{n})| = n^{-2} 
\]
and $|\widehat f(z)| = 0$ for integers $z \in \mathbb{Z}$ which are not of the form $z=2^n$ for $n \geq 1$ an integer.  Moreover, for all $z = (z_1, \dots, z_d) \in \mathbb{Z}^d$
\[
|\widehat \mu(z)| = |\widehat f_d(z)| = |\widehat f(z_1)|   \cdots  |\widehat f(z_d)| .
\]
Therefore, if $z = (z_1, \dots, z_d) \in \mathbb{Z}^d$ is such that $z_i=2^{n_i}$  for integers $n_i \geq 1$, then
 \[
|\widehat \mu(z)| = n_1^{-2} \cdots n_d^{-2} 
\]
and $|\widehat \mu(z)| = 0$, otherwise. In particular,  $\fd \mu = 0$. Further, using the Fourier series representation \eqref{fouriercoeff} of the energy 
\[
\mathcal{J}_{s,1}(\mu) \approx 1 +  \sum_{z=-\infty}^{\infty} |\widehat{\mu}(z)|^2 |z|^{s-d} \lesssim  1 +  \sum_{n=1}^{\infty} n^{d-1} n^{-4}  2^{n(s-d)} <\infty
\]
for $s<d$ and so $\sd \mu =\dim_{\mathrm{F}}^1 \mu = d$.  Since $\fs \mu$ is concave, the result follows from Theorem \ref{cty0}.
\end{proof}

\subsection{Riesz products}

Proposition \ref{coeffs} gives useful information about the Fourier spectrum for $\theta$ which are the reciprocal of an integer.  However, if more information about the Fourier transform is available sometimes this can be pushed to all $\theta$.  This is the case   for Riesz products which we use below  to provide explicit examples with a more complicated Fourier spectrum than the examples we have met thus far.   Riesz products are a well-studied family of measures defined by 
\[
\mu_{a,\lambda} = \prod_{j=1}^\infty (1+a_j \cos(2\pi \lambda_j x)) \qquad (x \in [0,1])
\]
where  $a=(a_j)$ and $\lambda = (\lambda_j)$ with $a_j \in [-1,1]$ and $\lambda_j \in \mathbb{N}$ with $   \lambda_{j+1} \geq 3\lambda_j$   are given sequences.  Here $\mu_{a,\lambda}$ is the weak limit of the sequence of absolutely continuous measures associated to the truncated  products.   It is well-known that $\mu_{a,\lambda}$ is absolutely continuous if and only if $\sum_j a_j^2 < \infty$ and otherwise they are mutually singular with respect to Lebesgue measure, see \cite[Theorem 13.2]{mattila}.  The dimension theory of Riesz products, especially in the singular case, is well-studied, see \cite{hare, mattila}.  The Fourier coefficients of $\mu_{a,\lambda}$ can be derived easily giving
\begin{equation} \label{ftrp}
\widehat{\mu_{a,\lambda}}(m) = \Pi_{\eps_j \neq 0} (a_j/2)
\end{equation}
for integers $m \neq 0$ with (unique) representation
\[
m = \sum_j \eps_j \lambda_j \qquad (\eps_j \in \{-1,0,1\})
\]
and $\widehat{\mu_{a,\lambda}}(m) = 0$ for integers without such a representation. Therefore, for each $k \in \mathbb{N}$,
\[
\widehat{\mu_{a,\lambda}}(\lambda_k) = a_k/2
\]
and so 
\[
\fd \mu_{a,\lambda} \leq \liminf_{k \to \infty} \frac{-2 \log |a_k|}{\log \lambda_k}.
\]
Therefore, if $\fd \mu_{a,\lambda} >0$, then $\sum_j a_j^2 < \infty $ and $ \mu_{a,\lambda}$ is absolutely continuous. Moreover, if $\lambda_{j+1}/\lambda_{j} \to \infty$, then $\hd \mu_{a,\lambda} = 1$ and $\sd \mu_{a,\lambda} \geq 1$, see \cite[Corollary 3.3]{hare}.  

\begin{thm} \label{riesz1}
Suppose $ \lambda_{j+1} \leq C \lambda_j$  for some fixed $C \geq 3$ and  $\liminf_{k \to \infty} \frac{-2 \log |a_k|}{\log \lambda_k} = 0$. Let
\[
\mathcal{S}_\theta(s) =\sum_{k=1}^\infty \lambda_k^{s/\theta-1}  \, \Pi_{j=1}^{k} (1+|a_j|^{2/\theta} 2^{1-2/\theta})  .
\]
Then, for $\theta \in (0,1]$,
\[
\fs \mu_{a,\lambda} =    \sup\{s \leq  \theta : \mathcal{S}_\theta(s)  < \infty\}.
\]
\end{thm}

\begin{proof}
We first prove the upper bound.  Note that $\fd \mu_{a,\lambda} =  \liminf_{k \to \infty} \frac{-2 \log |a_k|}{\log \lambda_k} = 0$ and so $\fs \mu_{a,\lambda} \leq \theta$ by Corollary  \ref{bounds}.  Observe that
\[
|\widehat{\mu_{a,\lambda}}(z)| \geq   |\widehat{\mu_{a,\lambda}}(m)|/2
\]
for all $z \in \mathbb{R}$ with $|z-m| \leq 1/2$ and where $m \neq 0$ has a (unique) representation
\begin{equation} \label{unirep}
m = \sum_j \eps_j \lambda_j \qquad (\eps_j \in \{-1,0,1\}).
\end{equation}
Summing over such $m$,
\begin{equation} \label{easylower2} 
\mathcal{J}_{s,\theta}(\mu_{a,\lambda})^{1/\theta}   \gtrsim       \sum_{m} |\widehat{\mu_{a,\lambda}}(m)|^{2/\theta} |m|^{s/\theta-1}.
\end{equation}
Then, by following the proof of \cite[Theorem 13.3]{mattila}, one obtains that the right hand side is finite if and only if $\mathcal{S}_\theta(s)  < \infty$.  For completeness we include the argument.  For $m \neq 0$ with unique representation \eqref{unirep}, let $k=k_m$ be the largest $j$ for which $\eps_j \neq 0$ and observe that $|m| \approx \lambda_k$.  Then, summing over such maximal $k$ instead of $m$, writing $b_j=2^{1/2-1/\theta}a_j^{1/\theta}$ and applying \eqref{ftrp},
\begin{align*}
  \sum_{m} |\widehat{\mu_{a,\lambda}}(m)|^{2/\theta} |m|^{s/\theta-1} & \approx \sum_{k=1}^\infty \sum_{j_1 < \cdots < j_l = k} 2^l \left(\Pi_{i=1}^l \frac{a_{j_i}}{2} \right)^{2/\theta} \lambda_k^{s/\theta-1}\\
& = \sum_{k=1}^\infty \sum_{j_1 < \cdots < j_l = k} \left(\Pi_{i=1}^l  b_{j_i} \right)^{2} \lambda_k^{s/\theta-1}\\
& = \sum_{k=1}^\infty b_k^2(1+b_1^2) \cdots (1+b_{k-1}^2) \lambda_k^{s/\theta-1}
\end{align*}
and finiteness of the last sum is easily seen to be equivalent to 
\[
\sum_{k=1}^\infty    \lambda_k^{s/\theta-1}\Pi_{j=1}^k (1+ b_j^2) =  \mathcal{S}_\theta(s)  < \infty.
\]
This proves that 
\begin{equation} \label{needed}
 \sum_{m} |\widehat{\mu_{a,\lambda}}(m)|^{2/\theta} |m|^{s/\theta-1} < \infty \quad \Leftrightarrow \quad \mathcal{S}_\theta(s)  < \infty
\end{equation}
as required.  In particular, by  \eqref{easylower2} and \eqref{needed}, the upper bound holds.  For the lower bound, observe that
\[
|\widehat{\mu_{a,\lambda}}(z)| \lesssim \sum_m |\widehat{\mu_{a,\lambda}}(m)|\min\{ 1 , |z-m|^{-1}\}
\]
for $z \in \mathbb{R}$ where the sum is over integers $m \neq 0$ with (unique) representation \eqref{unirep}. Fix $\theta \in (0,1]$, $0<s<\theta$ and $0<\eps< (\theta-s)/(2-\theta)$. Note that  $\mathcal{S}_\theta(s) = \infty$  for $s > \theta$. By Jensen's inequality,
\begin{align*}
|\widehat{\mu_{a,\lambda}}(z)|^{2/\theta}  &\lesssim_\eps \sum_m |\widehat{\mu_{a,\lambda}}(m)|^{2/\theta}\min\left\{ 1 , \frac{1}{|z-m|^{1-\eps(2/\theta-1)}}\right\}.
\end{align*}
Then,  applying Fubini's theorem,
\begin{align*}
\mathcal{J}_{s,\theta}(\mu_{a,\lambda})^{1/\theta}  &=\int_\mathbb{R} |\widehat{\mu_{a,\lambda}}(z)|^{2/\theta} |z|^{s/\theta-1} \, dz\\ 
 &\lesssim_\eps \sum_m |\widehat{\mu_{a,\lambda}}(m)|^{2/\theta}\int_\mathbb{R} \min\left\{ |z|^{s/\theta-1}  , \frac{|z|^{s/\theta-1} }{|z-m|^{1-\eps(2/\theta-1)}}\right\} \, dz  \\
 &\lesssim \sum_m |\widehat{\mu_{a,\lambda}}(m)|^{2/\theta} \bigg(\int_{|z-m| \leq 1}  |m|^{s/\theta-1}  \, dz + \int_{1 \leq |z-m| \leq m/2}  \frac{|m|^{s/\theta-1} }{|z-m|^{1-\eps(2/\theta-1)}} \, dz \\
&\, \qquad \quad   + \int_{ m/2 \leq |z-m| \leq 2m}  \frac{|z|^{s/\theta-1}}{|m|^{1 -\eps(2/\theta-1)}} \, dz  + \int_{  |z-m| \geq 2m}  \frac{1}{|z-m|^{2-s/\theta -\eps(2/\theta-1)}} \, dz \bigg)\\
 &\lesssim \sum_m |\widehat{\mu_{a,\lambda}}(m)|^{2/\theta} \bigg(  |m|^{s/\theta-1}  + |m|^{s/\theta-1} |m|^{\eps(2/\theta-1)} \\
&\, \qquad \qquad \qquad  +|m|^{-1+\eps(2/\theta-1)}  |m|^{s/\theta }  +  | m|^{-1+s/\theta +\eps(2/\theta-1)}  \bigg)\\
&\lesssim \sum_m |\widehat{\mu_{a,\lambda}}(m)|^{2/\theta}  |m|^{s/\theta-1+\eps(2/\theta-1)} \\
& <\infty
\end{align*}
provided $s +\eps(2-\theta)< \sup\{s \leq  \theta : \mathcal{S}_\theta(s)  < \infty\}$ using the equivalence \eqref{needed}.  Since $\eps>0$ can be chosen arbitrarily small, the lower bound follows.  
\end{proof}

We note a pleasant explicit formula  in the following simple case.

\begin{cor} \label{riesz2}
If $\lambda_j=\lambda^j \geq 3$ and $a_j=a \in [-1,1]$ with $a \neq 0$ and $\lambda\geq 3$ constants, then,  for $\theta \in [0,1]$, 
\[
\fs \mu_{a,\lambda} = \theta-\theta \frac{\log(1+|a|^{2/\theta}2^{1-2/\theta})}{\log \lambda}.
\]
\end{cor}


\subsection{An example with discontinuity at $\theta=0$}

Here we construct a (necessarily unbounded) measure for which the Fourier spectrum is discontinuous at $\theta = 0$.  The construction is similar to those above, but utilises the unbounded domain to achieve rapid Fourier decay around isolated peaks.  

\begin{prop} \label{discont}
There exists a finite Borel measure on $\mathbb{R}$ for which $\fs \mu$ is not continuous at $\theta=0$.
\end{prop}

\begin{proof}
Define $f: \mathbb{R} \to [0,\infty)$ by
\[
f(x) = \sum_{n=1}^\infty n^{-2} n^{-n}\left(2+ \sin(2\pi2^nx) \right) \textbf{1}_{[0,n^n]}(x)
\]
where $\textbf{1}_{[0,n^n]}$ is the indicator function on $[0,n^n]$. Define an (unbounded) measure $\mu$ by $\mu = f dx$ noting that
\[
\mu(\mathbb{R}) \leq  \sum_n 3(n^{-2} n^{-n}) n^n = 3\sum_n n^{-2}   < \infty.
\]
For integers $n \geq 1$,
\[
|\widehat \mu(2^n) | \approx  (n^{-2} n^{-n})n^n = n^{-2}
\]
and so $\fd \mu = \dim_\textup{F}^0 \mu = 0$.  Moreover, for all $z \in \mathbb{R}$,
\[
|\widehat \mu(z) | \lesssim \max_n \min\left\{ n^{-2}  , \frac{1}{n^n  |z-2^n|} \right\} \lesssim \begin{cases} 
      n^{-2} & \frac{n^n 2^n}{n^n+1} \leq |z| \leq \frac{n^n 2^n}{n^n-1} \text{ (for some $n \geq 3$)} \\
      |z|^{-1}& \text{otherwise}
   \end{cases}
\]
Therefore, for $\theta \in (0,1]$ and $s >0$,
\begin{align*}
\mathcal{J}_{s, \theta} (\mu)^{1/\theta} &\lesssim  \int   |z|^{-2/\theta} |z|^{s/\theta-1} \, dz  +  \sum_n n^{-4/\theta}  \int_{\frac{n^n 2^n}{n^n+1} \leq |z| \leq \frac{n^n 2^n}{n^n-1}}     |z|^{s/\theta-1} \, dz  \\
&\lesssim   \int   |z|^{(s-2)/\theta-1} \, dz  +   \sum_n n^{-4/\theta}     2^{n(s/\theta-1)}2^n n^{-n}  \\
&<\infty
\end{align*}
provided $s<2$.  Therefore, $\fs \mu \geq 2$ for $\theta \in (0,1]$ and $\fs \mu$ is not continuous at $\theta = 0$. 
\end{proof}

It is easy to adapt the above calculation to show that $\fs \mu = 2$ for $\theta \in (0,1]$.  Further, this example can be modified easily to obtain different behaviour at the discontinuity, including positive Fourier dimension and arbitrarily large jumps.

\section{Connection to Strichartz  bounds and average Fourier dimensions}

Strichartz \cite{stric1, stric2} considered bounds for  averages of the  Fourier transform of the form
\[
R^{d-\beta_k} \lesssim \int_{|z| \leq R} |\widehat \mu (z)|^{2k} \, dz \lesssim R^{d-\alpha_k}
\]
for integers $k \geq 1$ and $0 \leq \alpha_k \leq \beta_k$. Motivated by this, for $\theta \in (0,1]$, let
\[
\overline{F}_\mu(\theta) = \limsup_{R \to \infty} \frac{\theta \log \left(R^{-d} \int_{|z| \leq R} |\widehat \mu (z)|^{2/\theta} \, dz \right)}{-\log R}
\]
and
\[
\underline{F}_\mu(\theta) = \liminf_{R \to \infty} \frac{\theta \log \left(R^{-d} \int_{|z| \leq R} |\widehat \mu (z)|^{2/\theta}\, dz \right)}{-\log R}.
\]
Equivalently, $\overline{F}_\mu(\theta)$ is the infimum of $\beta \geq 0$ for which
\[
R^{d-\beta/\theta} \lesssim \int_{|z| \leq R} |\widehat \mu (z)|^{2/\theta} \, dz
\]
and  $\underline{F}_\mu(\theta)$ is the supremum  of $\alpha\geq 0$ for which
\[
  \int_{|z| \leq R} |\widehat \mu (z)|^{2/\theta} \, dz\lesssim R^{d-\alpha/\theta}.
\]
One can interpret $\underline{F}_\mu(\theta)$ as a `$\theta$-averaged Fourier dimension'. Note that
\[
 \int_{\rd } |\widehat \mu (z)|^{2/\theta} \, dz= \lim_{R \to \infty} \int_{|z| \leq R} |\widehat \mu (z)|^{2/\theta}\, dz
\]
exists and is either a positive finite number or $+\infty$.   In the former case $\underline{F}_\mu(\theta) = \overline{F}_\mu(\theta) = \theta d$ and in the latter case   $0 \leq \underline{F}_\mu(\theta) \leq \overline{F}_\mu(\theta) \leq \theta d$.  There is a connection between the Fourier spectrum and these average Fourier dimensions.

\begin{thm} \label{stric3}
Let $\mu$ be a finite Borel measure on $\mathbb{R}^d$ and $\theta \in (0,1]$.  Then
\[
\fs \mu \geq   \underline{F}_\mu(\theta).
\]
\end{thm}

\begin{proof}
We may assume $\underline{F}_\mu(\theta) >0$.  Let $c>1$  and $0<s <\alpha<\underline{F}_\mu(\theta).$  Then
\begin{align*}
\mathcal{J}_{s,\theta}(\mu)^{1/\theta} = \int_{\mathbb{R}^d} |\widehat \mu(z)  |^{2/\theta} |z|^{s/\theta-d} \, dz 
&  \lesssim 1+ \sum_{k =1}^\infty  \int_{c^{k-1} <|z| \leq c^k} |\widehat \mu(z)  |^{2/\theta} |z|^{s/\theta-d} \, dz  \\
&  \approx 1+ \sum_{k =1}^\infty   c^{k(s/\theta-d)} \int_{c^{k-1} <|z| \leq c^k} |\widehat \mu(z)  |^{2/\theta}  \, dz  \\
&  \leq   1+ \sum_{k =1}^\infty c^{k(s/\theta-d)}    \int_{|z| \leq c^k} |\widehat \mu(z)  |^{2/\theta} \, dz  \\
&  \lesssim  1+ \sum_{k =1}^\infty c^{k(s/\theta-d)}    c^{k(d-\alpha/\theta)}  \\
&< \infty
\end{align*}
proving $\fs \mu \geq \alpha$, which proves the result.
\end{proof}

The connection becomes stronger, and in fact $\fs \mu$ and $\underline{F}_\mu(\theta)$ coincide, in the following special case.

\begin{thm} \label{stric1}
Let $\mu$ be a finite Borel measure on $\mathbb{R}^d$ and $\theta \in (0,1]$. If $ \underline{F}_\mu(\theta) <  \theta d$, then
\[
\fs \mu =   \underline{F}_\mu(\theta).
\]
\end{thm}

\begin{proof}
The lower bound $\fs \mu \geq   \underline{F}_\mu(\theta)$ comes from Theorem  \ref{stric3} and so we prove the upper bound.  Choose $\underline{F}_\mu(\theta)  <\beta < s< \theta d$.  Then there exist arbitrarily large $R>0$ satisfying
\[
 \int_{|z| \leq R} |\widehat \mu(z)  |^{2/\theta} \, dz  \gtrsim R^{d-\beta/\theta}
\]
and
\begin{align*}
\mathcal{J}_{s,\theta}(\mu)^{1/\theta} = \int_{\mathbb{R}^d} |\widehat \mu(z)  |^{2/\theta} |z|^{s/\theta-d} \, dz 
&  =  \sup_{R>0} \int_{|z| \leq R} |\widehat \mu(z)  |^{2/\theta} |z|^{s/\theta-d} \, dz  \\
&  \geq \sup_{R>0}R^{s/\theta-d}  \int_{|z| \leq R} |\widehat \mu(z)  |^{2/\theta} \, dz  \\
&  \gtrsim \sup_{R>0}R^{s/\theta-d}  R^{d-\beta/\theta}  \\
&=\infty
\end{align*}
which proves $\fs \mu \leq \beta$, proving the result.
\end{proof}

\subsection{Self-similar measures}

Combining \cite[Theorem 2]{bisbas} for $\theta \in (0,1)$ and   \cite[Theorem 4.4]{stric1} for $\theta=1$ with  Theorem  \ref{stric3} we get the following.
 
\begin{cor} \label{selfsim1}
For $p \in (0,1)$ with $p \neq 1/2$, let  $\mu_p$ be the self-similar measure on $[0,1]$ given by the distribution of the random series
\[
\sum_{n=0}^\infty X_n (1/2)^n
\]
where $\mathbb{P}(X_n=0)=p$ and $\mathbb{P}(X_n=1)=(1-p)$.  Then
\[
\fs \mu_p \geq  \underline{F}_{\mu_p}(\theta)  \geq  \theta - \theta \log_2(1+|2p-1|^{2/\theta})
\]
for all $\theta \in (0,1)$ and
\[
\sd \mu_p = \underline{F}_{\mu_p}(1)  = 1 -  \log_2(1+|2p-1|^{2}).
\]
In particular, $\fs \mu_p>\theta \sd \mu_p$ for all $\theta \in (0,1)$,  
\[
\lim_{\theta \to 0} \frac{\fs \mu_p}{\theta} \geq 1
\]
and
\[
\sd \mu_p < \hd \mu_p = \frac{p \log p + (1-p)\log (1-p)}{-\log 2}.
\]
\end{cor}

In the above, if $p=1/2$, then $\mu_p$ is Lebesgue measure restricted to $[0,1]$ and 
\[
\fs \mu_p = \fd \mu_p = \sd \mu_p = 1
\]
for all $\theta \in (0,1)$.  We can also apply Strichartz' work \cite{stric1, stric2} to get some partial information for other self-similar measures. For example, we get some non-trivial information about the Fourier spectrum of self-similar measures on the middle third Cantor set.  Similar estimates, sometimes for more choices of $\theta$, can also be deduced from \cite{stric1, stric2} for other self-similar measures but we leave the details to the reader.

\begin{cor} \label{selfsim2}
For $p \in (0,1)$, let  $\mu_p$ be the self-similar measure on the middle third Cantor set corresponding to Bernoulli weights $p, (1-p)$.    Then
\[
\fd \mu = \dim_\textup{F}^0 \mu = 0
\]
\[
 \dim_\textup{F}^{1/2} \mu = \frac{\log (p^4+4p^2(1-p)^2+(1-p)^4)}{-2\log 3}
\]
and
\[
\sd \mu = \dim_\textup{F}^1 \mu =  \frac{\log (p^2+(1-p)^2)}{-\log 3} < 2  \dim_\textup{F}^{1/2} \mu.
\]
In particular, $\fs \mu$ is not a linear function.
\end{cor}

\begin{proof}
The formulae for $ \dim_\textup{F}^{1/2} \mu$ and $ \dim_\textup{F}^{1} \mu$ come from  Theorem \ref{stric3} combined with \cite[Corollary 4.4]{stric1}.  The fact that $\fd \mu = 0$ follows from the well-known and easily proved fact that the Fourier dimension of the middle third Cantor set is 0.
\end{proof}

 It would be interesting to investigate the Fourier spectrum of self-similar measures more generally and also for other dynamically invariant.  In particular, there is a lot of interest currently on the Fourier dimension of invariant measures in various contexts, see for example \cite{li,  stevens, solomyak}.

\begin{ques}
What is $\fs \mu$ when  $\mu$ is a self-similar measure on the middle third Cantor set? What about more general self-similar measures, self-affine measures and other dynamically invariant measures?
\end{ques}


\section{Another example: measures on curves}

To bolster our collection of examples, here  we provide a simple family  where the Fourier spectrum can be computed explicitly and exhibits some non-trivial behaviour. Let $p >1$ and let $\mu_p$ be the lift of Lebesgue measure on $[0,1]$ to the curve $\{(x,x^p) : x \in [0,1]\} \subseteq \mathbb{R}^2$ via the map $x \mapsto (x,x^p)$.

\begin{thm} \label{xp}
For $p >1$ and $\theta \in [0,1]$,
\[
\fs \mu_p = \min\{ 2/p+\theta (1-1/p), \, 1\}.
\]
\end{thm}

\begin{proof}
Noting
\[
\widehat{\mu_p}(z) = \int_0^1 e^{-2\pi i z \cdot (x, x^p)} \, dx,
\]
using polar coordinates $z = re^{i \alpha}$
\[
\mathcal{J}_{s,\theta}(\mu_p)^{1/\theta} =\int_{r=0}^\infty r^{s/\theta-1} \int_{\alpha=0}^{2\pi} \left\lvert \widehat{\mu_p}(z) \right\rvert ^{2/\theta} \, d\alpha \, d r .
\]
We split this integral up into three regions, which are handled separately. Define $\eps(p) \in (0,\pi/2)$ by  $\tan \eps(p)  = 1/(2p)$. 

First, for $\alpha$ such that $|\cos \alpha | \geq \eps(p)$, it is a simple consequence of van der Corput's lemma (see \cite[Theorem 14.2]{mattila}) that
\[
\left\lvert \widehat{\mu_p}(z) \right\rvert  \lesssim r^{-1/2} = |z|^{-1/2}
\]
and so 
\begin{equation} \label{estxp1}
\int_{r=0}^\infty r^{s/\theta-1} \int_{|\cos \alpha | \geq \eps(p)} \left\lvert \widehat{\mu_p}(z) \right\rvert ^{2/\theta} \, d\alpha \, d r < \infty
\end{equation}
for $s<1$.

Second,  for $z=re^{i\alpha}$ such that $r \geq 10$ and $r^{-(p-1)/p}\leq |\cos \alpha| \leq \eps(p) $, 
\begin{align*}
\left\lvert \widehat{\mu_p}(z) \right\rvert   &=\left\lvert  \int_0^1 e^{-2\pi i(rx\cos\alpha+ rx^p\sin\alpha)} \, dx\right\rvert = \left\lvert  \int_0^1 e^{2\pi  i r|\cos\alpha|\phi(x)} \, dx\right\rvert 
\end{align*}
for 
\[
\phi(x) = - \frac{\cos \alpha}{|\cos \alpha|} x -  \frac{\sin \alpha}{|\cos \alpha|}x^p.
\]
 Since
\[
|\phi'(x)|  \geq 1-p \tan \eps(p)  = 1/2,
\]
van der Corput's lemma (see \cite[Theorem 14.2]{mattila})  gives
\[
\left\lvert \widehat{\mu_p}(z) \right\rvert    \lesssim   \frac{1}{r | \cos \alpha |}.
\]
Therefore,
\begin{align}\label{estxp2}
& \hspace{-10mm} \int_{r=10}^\infty r^{s/\theta-1} \int_{r^{-(p-1)/p}\leq |\cos \alpha| \leq \eps(p)} \left\lvert \widehat{\mu_p}(z) \right\rvert ^{2/\theta} \, d\alpha \, d r  \nonumber \\
&\lesssim  \int_{r=10}^\infty r^{s/\theta-1-2/\theta} \int_{\alpha=0}^{\arccos(r^{-(p-1)/p})} \left(\frac{1}{|\cos \alpha|}\right)^{2/\theta} \, d\alpha \, d r  \nonumber \\
&\lesssim \int_{r=10}^\infty  r^{s/\theta-1-2/\theta} r^{\left(\frac{2}{\theta}-1\right)\left(\frac{p-1}{p}\right)}\, d r  \nonumber \\
&<\infty
\end{align}
provided $s <2/p+\theta (1-1/p)$.

Finally, for   $z=re^{i\alpha}$ such that $r \geq 10$  and $0 \leq |\cos \alpha| \leq r^{-(p-1)/p} $ and  using the substitution $y= rx\cos\alpha+ rx^p\sin\alpha$,
\begin{align*}
\left\lvert \widehat{\mu_p}(z) \right\rvert   &=\left\lvert  \int_0^1 e^{-2\pi i(rx\cos\alpha+ rx^p\sin\alpha)} \, dx\right\rvert \\
 & =   \left\lvert \int_0^{r(\cos\alpha +\sin\alpha)} \frac{e^{-2\pi i y}}{r\cos\alpha+rp x^{p-1}\sin \alpha}  \, dy\right\rvert  \\
 & \lesssim  \left\lvert \int_0^{r|\cos\alpha |^{\frac{p}{p-1}}} \frac{e^{-2\pi i y}}{r\cos\alpha}  \, dy \right\lvert + \left\lvert \int_{r|\cos\alpha|^{\frac{p}{p-1}}}^{r(\cos\alpha +\sin\alpha)} \frac{e^{-2\pi i y}}{rp x^{p-1}\sin \alpha}  \, dy\right\rvert \\
  & \lesssim   \frac{ r|\cos\alpha |^{\frac{p}{p-1}}}{r|\cos\alpha|}  + \frac{1}{r^{1/p}}\left\lvert \int_{ r|\cos\alpha|^{\frac{p}{p-1}}}^{r(\cos\alpha +\sin\alpha)}e^{-2\pi i y} y^{(1-p)/p}  \, dy\right\rvert \\
  & \lesssim   |\cos\alpha |^{{\frac{1}{p-1}}}  + \frac{1}{r^{1/p}}  \\
  & \lesssim   \frac{1}{r^{1/p}} 
\end{align*}
Therefore,
\begin{align}\label{estxp3}
& \hspace{-10mm} \int_{r=10}^\infty r^{s/\theta-1} \int_{0 \leq |\cos \alpha| \leq r^{-(p-1)/p}} \left\lvert \widehat{\mu_p}(z) \right\rvert ^{2/\theta} \, d\alpha \, d r  \nonumber \\
&\lesssim  \int_{r=10}^\infty r^{s/\theta-1} \int_{\alpha=\arccos(r^{-(p-1)/p})}^{\pi/2}  \left( \frac{1}{r^{1/p}} \right)^{2/\theta} \, d\alpha \, d r  \nonumber \\
&\lesssim \int_{r=10}^\infty r^{s/\theta-1-\frac{2}{\theta p}} r^{-(p-1)/p} \, d r  \nonumber \\
&<\infty
\end{align}
provided $s <2/p+\theta (1-1/p)$.  Together, \eqref{estxp1}, \eqref{estxp2} and \eqref{estxp3} establish the desired lower bound noting that the integral over $|z| \leq 10$ is trivially finite.

For the upper bound, let $1 >q>(p-1)/p$.  Then, similar to above, for $z$ such that $0 \leq |\cos \alpha| \leq r^{-q} $, 
\begin{align*}
\left\lvert \widehat{\mu_p}(z) \right\rvert    & =   \left\lvert \int_0^{r(\cos\alpha +\sin\alpha)} \frac{e^{-2\pi i y}}{r\cos\alpha+rp x^{p-1}\sin \alpha}  \, dy\right\rvert  \\
 & \geq   \left\lvert \int_{r|\cos\alpha|^{1/q}}^{r(\cos\alpha +\sin\alpha)} \frac{e^{-2\pi i y}}{rp x^{p-1}\sin \alpha}  \, dy\right\rvert \ - \   \left\lvert \int_0^{r|\cos\alpha |^{1/q}} \frac{e^{-2\pi i y}}{r\cos\alpha}  \, dy \right\lvert\\
  & \geq   \frac{1}{r^{1/p}}\left\lvert \int_{ r|\cos\alpha|^{1/q}}^{r(\cos\alpha +\sin\alpha)}e^{-2\pi i y} y^{(1-p)/p}  \, dy\right\rvert \ - \   \frac{ r|\cos\alpha |^{1/q}}{r|\cos\alpha|}  \\
  & \gtrsim   \frac{1}{r^{1/p}} -  |\cos\alpha |^{{1/q-1}}    \\
  & \geq  \frac{1}{2 r^{1/p}}  
\end{align*}
for $r \geq r_0$ for some constant $r_0 \approx_q 1$.  Therefore,
\begin{align*}
& \hspace{-10mm} \int_{r=0}^\infty r^{s/\theta-1} \int_{0 \leq |\cos \alpha| \leq r^{-(p-1)/p}} \left\lvert \widehat{\mu_p}(z) \right\rvert ^{2/\theta} \, d\alpha \, d r  \nonumber \\
&\gtrsim  \int_{r=r_0}^\infty r^{s/\theta-1} \int_{\alpha=\arccos(r^{-q})}^{\pi/2}  \left( \frac{1}{r^{1/p}} \right)^{2/\theta} \, d\alpha \, d r  \nonumber \\
&\gtrsim \int_{r=r_0}^\infty r^{s/\theta-1-\frac{2}{\theta p}} r^{-q} \, d r  \nonumber \\
&=\infty
\end{align*}
provided $s \geq 2/p+\theta q$. Combined with the trivial fact that $\fs \mu_p \leq \sd \mu_p \leq 1$, this proves the desired upper bound by letting $q \to (p-1)/p$.
\end{proof}


\section{Convolutions and sumsets}

Given sets $X, Y \subseteq \rd$, the \emph{sumset} of $X$ and $Y$  is 
\[
X+Y = \{ x+y : x \in X, y \in Y\} \subseteq \rd.
\]
Such sets arise naturally in many contexts and a question of particular interest in additive combinatorics is to understand, for example,  how the `size' of $X+X$ is related to the `size' of $X$.  There is interest in determining conditions which ensure   $\hd (X+X) > \hd X$ or perhaps $\hd (kX) \to d$ as $k \to \infty$ where $kX$ is the $k$-fold sumset for integers $k \geq 1$, see for example \cite{linden}.  Convolutions are natural measures to consider in this context: if  $\mu$ and $\nu$ are measures on $X$ and $Y$, respectively, the convolution $\mu*\nu$ is supported on the sumset $X+Y$.   It turns out that the Fourier spectrum precisely characterises when the Sobolev dimension increases under convolution, see Corollary \ref{conv2}, and gives many partial results about the dimensions of sumsets and convolutions more generally.

\subsection{Convolutions}

First we give some general lower bounds for the Fourier spectrum of a convolution. The  special case when $d = \theta= \lambda= 1$,  $\sd \mu = 1$ and $\fd \nu>0$ is essentially \cite[Lemma 2.1 (1)]{shmerkin}. 

\begin{thm} \label{convolution}
Let  $\mu$ and $\nu$ be finite Borel measures on $\mathbb{R}^d$.  Then for all $s,t \geq 0$  
\[
\mathcal{J}_{s+t,\theta}(\mu * \nu) \lesssim \inf_{\lambda \in [0,1]} \mathcal{J}_{s,\lambda\theta}(\mu ) \mathcal{J}_{t,(1-\lambda)\theta}(\nu ).
\]
In particular, 
\[
\fs (\mu * \nu) \geq \sup_{\lambda \in [0,1]} \left( \dim_\mathrm{F}^{\lambda \theta } \mu +  \dim_\mathrm{F}^{(1-\lambda)\theta}  \nu \right).
\]
\end{thm}

\begin{proof}
Let $\lambda \in [0,1]$ and   $s,t > 0$.  Let  $p=1/\lambda \in [1, \infty]$ and $q=1/(1-\lambda) \in [1,\infty]$ be H\"older conjugates and  write $d m_d =\min\{|z|^{-d}, 1\} \, dz$.  Then
\begin{align*}
\mathcal{J}_{s+t,\theta}(\mu * \nu)^{1/\theta} & \approx \int_{\mathbb{R}^d} |\widehat{\mu * \nu}(z)  |^{2/\theta} |z|^{(s+t)/\theta} \, d m_d(z) \\
& = \int_{\mathbb{R}^d} |\widehat{\mu }(z)  |^{2/\theta} |z|^{s/\theta} \, |\widehat{ \nu}(z)  |^{2/\theta} |z|^{t/\theta}  \,  d m_d(z)  \qquad \text{(by convolution formula)} \\
& \leq \left( \int_{\mathbb{R}^d} |\widehat{\mu }(z)  |^{2p/\theta} |z|^{ps/\theta}  \,  d m_d(z) \right)^{1/p} \left( \int_{\mathbb{R}^d}  |\widehat{ \nu}(z)  |^{2q/\theta} |z|^{qt/\theta}  \,  d m_d(z)\right)^{1/q}  \\
&\hspace{5cm} \text{(by H\"older's inequality)}\\
&\approx \mathcal{J}_{s,\lambda\theta}(\mu )^{1/\theta}\mathcal{J}_{t,(1-\lambda)\theta}(\nu )^{1/\theta}  .
\end{align*}
This establishes the first claim.  It follows that for $\lambda \in [0,1]$ and all  $s < \dim_\mathrm{F}^{\lambda\theta} \mu$ and $t < \dim_\mathrm{F}^{(1-\lambda)\theta}  \nu$, $\fs (\mu * \nu) \geq s+t$, which proves second claim.
\end{proof}

One is often interested in how dimension grows (or how smoothness increases) under iterated convolution. For integers $k \geq 1$, we write $\mu^{*k}$ for the $k$-fold convolution of $\mu$ and $kX$ for the $k$-fold sumset.  The following is a generalisation of the simple fact that $\fd (\mu^{*k}) = k \fd \mu$.

\begin{lma} \label{conv1}
Let  $\mu$ be a finite  Borel measure on $\mathbb{R}^d$.  Then 
\[
\fs (\mu^{*k}) =   k \dim_\mathrm{F}^{\theta/k} \mu
\]
for all $\theta \in [0,1]$ and all integers $k \geq 1$.  In particular,
\[
\sd (\mu^{*k}) =   k \dim_\mathrm{F}^{1/k} \mu
\]
and, provided $\fs \mu$ is continuous at $\theta=0$,
\[
\fd \mu = \lim_{k \to \infty} \frac{\sd (\mu^{*k})}{k}.
\]
\end{lma}

\begin{proof}
By the convolution formula
\[
 \int_{\mathbb{R}^d} \lvert\widehat{\mu^{*k}}(z)  \rvert^{2/\theta} |z|^{s/\theta-d} \, dz =  \int_{\mathbb{R}^d} |\widehat{\mu}(z)  |^{2/(\theta/k)} |z|^{(s/k)/(\theta/k)-d} \, dz 
\]
which proves the result.
\end{proof}

It is not possible to express $\fs (\mu * \nu) $ in terms of the Fourier dimension spectra of $\mu$ and $\nu$ in general if $\mu$ and $\nu$ are distinct.  For example, consider $\mu$ and $\nu$ given by 1-dimensional Hausdorff measure restricted to distinct unit line segments in the plane.  Then, by Corollary \ref{hyperplane},   $\fs \mu = \fs \nu = \theta$ for all $\theta$.  However, if the line segments are not contained in a common line, then the convolution $\mu * \nu$ is 2-dimensional Lebesgue measure restricted to a parallelogram and therefore $\fs (\mu * \nu) = 2$ for all $\theta$.  On the other hand, if the line segments are contained in a common line, then  $\fs (\mu * \nu) = \theta$ for all $\theta$.

Lemma \ref{conv1} shows that  the Fourier dimension of a measure can be expressed in terms of the Sobolev dimension of convolutions of the measure with itself.  This may be of use in applications since the Fourier dimension is usually harder to compute than the Sobolev dimension. Recall that continuity of $\fs \mu$ at $\theta=0$ is a very mild assumption and holds, for example, provided $|z|^\alpha \in L^1(\mu)$ for some $\alpha>0$, see Lemma \ref{holder} and Theorem \ref{cty0}.

Lemma \ref{conv1} plus concavity (Theorem \ref{concave})  imply that the Fourier spectrum cannot decrease under convolution.  However, we can say much more and in fact the Fourier spectrum necessarily  increases unless it has a very restricted form.  We also get a precise characterisation of when the Sobolev dimension increases under convolution.  

\begin{cor} \label{conv2}
Let  $\mu$ be a  finite Borel measure on $\mathbb{R}^d$ and $\theta \in (0,1]$.  Then
\[
\fs(\mu * \mu) > \fs \mu
\]
if and only if $\dim_\mathrm{F}^{\lambda \theta} \mu > \lambda  \fs \mu$ for some $\lambda \in [0,1)$. In particular,
\[
\sd(\mu * \mu) > \sd \mu
\]
 if and only if  $\dim_\mathrm{F}^{\lambda} \mu > \lambda \sd \mu$ for some $\lambda \in [0,1)$.
\end{cor}

\begin{proof}
One direction is Lemma \ref{conv1}.  Indeed, if  $\dim_\mathrm{F}^{\lambda \theta} \mu = \lambda  \fs \mu$ for all $\lambda \in [0,1)$, then $\fs (\mu * \mu) =  2\dim_\mathrm{F}^{\theta/2} \mu = \fs \mu$.  

To prove the other direction, let $\lambda \in [0,1)$ be such that $\dim_\mathrm{F}^{\lambda \theta} \mu > \lambda  \fs \mu$.  By Theorem \ref{concave}, $\fs \mu$ is concave and therefore $ \dim_\mathrm{F}^{\theta/2}  \mu>  (\fs \mu)/2$.  Then, by Lemma \ref{conv1},
\[
\fs (\mu * \mu) =  2\dim_\mathrm{F}^{\theta/2} \mu > \fs \mu.  
\]
The special case concerning  Sobolev dimension is obtained by setting $\theta=1$.
\end{proof}

The previous result characterises when the Sobolev dimension increases under convolution.  In fact, using Lemma \ref{conv1}, the Fourier spectrum also characterises the limiting behaviour of the Sobolev dimension of iterated convolutions.

\begin{cor} \label{iterated}
Let  $\mu$ be a   finite  Borel measure on $\mathbb{R}^d$ such that $\widehat \mu$ is $\alpha$-H\"older.  Then
\[
\lim_{k \to \infty}  \left( \sd (\mu^{*k}) - k \fd \mu \right) =  \partial_+ \fs \mu \vert_{\theta=0}
\]
where   $D = \partial_+ \fs \mu \vert_{\theta=0}$ is the right semi-derivative of $\fs \mu$ at $0$.  Moreover, by Theorem  \ref{cty0}
\[
 \sd \mu- \fd \mu \leq D \leq d\left(1+\frac{\fd \mu}{2\alpha}\right).
\]
In particular, if $\fd \mu>0$, then $\sd (\mu^{*k}) \sim k \fd \mu$ and if $\fd \mu=0$, then $\sd (\mu^{*k})$ is a bounded monotonic  sequence which converges to the right semi-derivative of $\fs \mu$ at $0$.
\end{cor}

\subsection{Sumsets and iterated sumsets}

Next we give sufficient conditions for the Hausdorff dimension of $X+Y$ to exceed the Hausdorff dimension of $Y$.

\begin{cor} \label{X+Y}
Let $X, Y \subseteq \mathbb{R}^d$ be non-empty sets with $Y$ Borel.  If  
\[
\lambda \hd Y< \dim_\mathrm{F}^{\lambda} X \leq d- (1-\lambda) \hd Y
\]
 for some $\lambda \in [0,1)$, then
\[
\hd(X+Y) \geq \hd Y +(\dim_\mathrm{F}^{\lambda} X-\lambda \hd Y) > \hd Y.
\]
If
\[
\dim_\mathrm{F}^{\lambda} X > d- (1-\lambda) \hd Y
\]
 for some $\lambda \in [0,1)$, then $X+Y$ has positive $d$-dimensional Lebesgue measure and if $X$ supports a measure $\mu$ with
\[
\dim_\mathrm{F}^{\lambda} \mu > 2d- (1-\lambda) \hd Y
\]
 for some $\lambda \in [0,1)$, then $X+Y$ has non-empty interior.
\end{cor}

\begin{proof}
Let $\eps>0$.  By definition of $\fs X$, there exists  a finite Borel measure $\mu$ supported on a closed set $X_0 \subseteq X$ with $\dim_\mathrm{F}^{\lambda} \mu \geq  \dim_\mathrm{F}^{\lambda} X -\eps$.  Moreover, there exists a finite Borel measure $\nu$ supported on a closed set $Y_0 \subseteq Y$ with $\sd \nu \geq  \hd Y - \eps$ and therefore, by concavity of the Fourier spectrum for measures, $\dim_\mathrm{F}^{\theta} \nu \geq \theta (\hd Y-  \eps )$ for all $\theta \in (0,1)$.  Then $\mu * \nu $ is  supported on $X_0+Y_0 \subseteq X+Y$ and, by Theorem \ref{convolution},
\begin{align*}
\sd (\mu * \nu ) &\geq \dim_\mathrm{F}^\lambda \mu + \dim_\mathrm{F}^{(1-\lambda)} \nu \\
& \geq  \dim_\mathrm{F}^{\lambda} X  -\eps   + (1-\lambda)\left(\hd Y-\eps \right) \\
 &=  \hd Y +(\dim_\mathrm{F}^{\lambda} X-\lambda \hd Y)   -\eps(2-\lambda) 
\end{align*}
which proves the first two claims upon letting $\eps \to 0$.  The final claim (giving non-empty interior) is proved similarly by establishing that $\sd (\mu * \nu )  > 2d$.  We omit the details. 
\end{proof}

If $X$ and $Y$ coincide, then the above result simplifies and we get a succinct  sufficient condition for dimension increase under addition.

\begin{cor} \label{sumsetcharacter}
Let $X \subseteq \mathbb{R}^d$ be a Borel set with $\hd X < d$. If   $\dim_\mathrm{F}^{\lambda} X > \lambda \hd X$ for some $\lambda \in [0,1)$, then
\[
\hd(X+X) > \hd X.
\]
\end{cor}

We note that  $\fs X$ does not characterise precisely  when $\hd(X+X) > \hd X$ (compare with Corollary \ref{conv2}). For example, let $X \subseteq \mathbb{R}^2$ be  the union of two unit line segments.  Then $\fs X = \theta$.  However, if the two line segments lie in a common line, then  $\hd (X+X) = \hd X = 1$, but if the two line segments do not lie in a common line, then $X+X$ has non-empty interior.

Finally, we consider dimension growth of the iterated sumset $kX$.  If $\fd X >0$, then  $kX$ will have non-empty interior for finite explicit $k$.  Therefore in the following we restrict to sets with $\fd X = 0$.  The Fourier spectrum gives a lower bound for the dimension growth in terms of the right semi-derivative at 0.
\begin{cor}
Let $X \subseteq \mathbb{R}^d$ be a non-empty set with $\fd X = 0$.    Then
\[
\lim_{k \to \infty} \hd (kX) \geq \sup_{\theta \in (0,1)} \frac{\fs X}{\theta} = \partial_+ \fs X \vert_{\theta=0}.
\]
\end{cor}

\begin{proof}
Fix $\theta \in (0,1)$, let $\eps>0$  and let $\mu$ be a finite Borel measure on $X$ with $\fs \mu > \fs X-\theta\eps$.  Then,  for integers $k \geq 1/\theta$, by Lemma \ref{conv1} and concavity of $\fs \mu$ (Theorem \ref{concave}),
\[
\hd (kX) \geq \sd (\mu^{*k}) =   k \dim_\mathrm{F}^{1/k} \mu \geq \frac{ \dim_\mathrm{F}^{\theta} \mu}{\theta} > \frac{ \fs X}{\theta} -\eps.
\]
Since this holds for all $\eps>0$ and all $\theta \in (0,1)$ and all sufficiently large $k$, the result follows.
\end{proof}

\subsection{Measures and sets which improve dimension}

Motivated by  the well-known problem of Stein to classify measures which are $L^p$-improving, one can ask when a measure $\mu$ will increase the dimension of all measures $\nu$ from some class under convolution simultaneously.  Theorem \ref{convolution} clearly gives some information about this.  We state one version, leaving further analysis to the reader.   See \cite{rossi} for a related question. The main interest here is in measures with Fourier dimension 0 since if $\mu$ has positive Fourier dimension then it increases the Sobolev dimension of all measures.

\begin{cor} \label{sobolevimproving}
Let $\mu$ be a finite Borel measure on $\mathbb{R}^d$ such that
\[
 \sup_{\theta \in (0,1)} \frac{\fs \mu}{\theta} \geq s.
\]
Then $\mu$ is Sobolev improving in the sense that for all finite Borel measures $\nu$ on $\rd$  with $\sd \nu < s$
\[
\sd (\mu * \nu)  > \sd \nu.
\] 
\end{cor}

\begin{proof}
Let $\nu$ be a finite Borel measure with $s> t>\sd \nu$. By assumption there exists  $\theta \in (0,1)$ such that 
\[
\fs \mu \geq t \theta.
\]
Then, by Theorem \ref{convolution}, $\sd (\mu * \nu) \geq  \dim_\mathrm{F}^{  \theta } \mu +  \dim_\mathrm{F}^{ 1-\theta}  \nu   \geq t \theta  + (1-\theta) \sd \nu > \sd \nu.$
\end{proof}

The utility of the previous result is that the Sobolev dimension of $\mu$ itself may be arbitrarily close to 0, that is, much smaller than $s$.  Next we state a version for sumsets, which follows immediately from Corollary \ref{X+Y}.  Again, the Hausdorff dimension of $X$ can be arbitrarily close to 0.

\begin{cor}
Let $X \subseteq \mathbb{R}^d$ be a non-empty set such that
\[
 \sup_{\theta \in (0,1)} \frac{\fs \mu}{\theta} \geq  s.
\]
Then $X$ is Hausdorff  improving in the sense that for all non-empty Borel $Y \subseteq \rd$  with $\hd Y<s$
\[
\hd (X+Y) > \hd Y.
\] 
\end{cor}

\section{Distance sets}

Given a  set $X \subseteq \rd$, the associated  \emph{distance set} is
\[
D(X) = \{ |x-y| : x,y \in X\}.
\]
Following \cite{distance}, there has been a lot of interest in the so-called `distance set problem', which is to relate the size of $D(X)$ with the size of $X$.  For example, it is conjectured that (for $d \geq 2$ and $X$ Borel)  $\hd X >d/2$ guarantees $\mathcal{L}^1(D(X)) >0$ and $\hd X \geq d/2$ guarantees $\hd D(X) = 1$.  Here $\mathcal{L}^1$ is the 1-dimensional Lebesgue measure on $\mathbb{R}$.  Both of these conjectures are open for all $d \geq 2$, despite much  recent interest and progress, e.g. \cite{guth, keletishmerkin,shmerkinwang}.  For example, the measure version of the conjecture holds for   Salem sets, that is, when $\fd X = \hd X >d/2$,  see \cite[Corollary 15.4]{mattila}, and the dimension version holds when $\pd X = \hd X \geq d/2$ \cite{shmerkinwang}. Here $\pd$ denotes the packing dimension, which is bounded below by the Hausdorff dimension in general.  We are able to obtain estimates for the size of  distance sets in terms of the Fourier spectrum, including the provision of new families of sets for which the above conjectures hold.  Our main result on distance sets is the following.  This result will follow from the  more general Theorem \ref{maindistance2} given below.
\begin{thm} \label{maindistance}
If $X \subseteq \rd$ satisfies 
\[
\sup_{\theta \in [0,1]} \left( \dim_\textup{F}^\theta X + \dim_\textup{F}^{1-\theta} X \right) > d,
\]
then  $\mathcal{L}^1(D(X)) >0$.  Otherwise
\[
\hd D(X) \geq 1-d+\sup_{\theta \in [0,1]}  \left( \dim_\textup{F}^\theta X + \dim_\textup{F}^{1-\theta} X \right) .
\]
\end{thm}

We note that it is possible for
\[
\sup_{\theta \in [0,1]} \left(\dim_\textup{F}^\theta X + \dim_\textup{F}^{1-\theta} X \right) > d,
\]
to hold  for sets with $\fd X = 0$ and $\hd X < \pd X$, even with $\hd X$ arbitrarily close to $d/2$. 

If we set  $\theta = 0$ in Theorem \ref{maindistance} then we find that if $X \subseteq \rd$ is a Borel set with $ \fd X+\hd X > d$, then  $\mathcal{L}^1(D(X)) >0$, with  a corresponding dimension bound.  These results were obtained by Mattila in \cite[Theorem 5.3]{mattiladistance}.  Note that this implies the measure (and dimension) version of the original conjecture for Salem sets.  Setting $\theta=1/2$ in Theorem \ref{maindistance} yields a pleasant corollary which appears more similar to the original conjecture.

\begin{cor}
Let  $X \subseteq \rd$.  If $\dim_\textup{F}^{1/2} X   > d/2$, then  $\mathcal{L}^1(D(X)) >0$ and, otherwise, $\hd D(X) \geq 1-d+ 2 \dim_\textup{F}^{1/2} X$.
\end{cor}

Since for $\theta \neq 1/2$ different values of $\theta$ are used simultaneously in Theorem  \ref{maindistance}, we are led naturally to consider \emph{two} measures supported on $X$, one for $\theta$ and one for $1-\theta$.  This leads us to consider mixed distance sets.  Given sets $X,Y \subseteq \rd$, the \emph{mixed distance set} of $X$ and $Y$ is
\[
D(X,Y) = \{ |x-y| : x \in X, y \in Y\}.
\]
Of course $D(X,X)$ recovers $D(X)$, but $D(X,Y)$ is (typically) a strict subset of $D(X \cup Y)$.  Theorem \ref{maindistance} follows from the following more general result which considers  mixed distance sets and directly uses the energies.

\begin{thm} \label{maindistance2}
Suppose  $\mu$ and $\nu$ are finite Borel measures on $\rd$ with
\[
\mathcal{J}_{s,\theta}(\mu) < \infty
\]
and
\[
\mathcal{J}_{t,1-\theta}(\nu) < \infty
\]
for some $s,t \geq 0$.  If $s+t \geq d$, then
\[
\mathcal{L}^1(D(\textup{spt}(\mu), \textup{spt}(\nu))) >0
\]
and if $s+t < d$, then    
\[
\hd D(\textup{spt}(\mu), \textup{spt}(\nu)) \geq 1-d+s+t.
\]
\end{thm}

We defer the proof of Theorem \ref{maindistance2} (as well as the simple deduction of Theorem \ref{maindistance}) to the following subsection. Another consequence of Theorem \ref{maindistance2} (choosing $\theta=1/2$) can be stated just in terms of the Fourier transform. Again, the assumption can be satisfied even when $\fd \mu = 0$.

\begin{cor} \label{cute}
If $\mu$ is a finite Borel  measure on $\rd$ with
\[
\int |\widehat \mu(z) |^4 \, dz <\infty
\]
then  $\mathcal{L}^1(D(\textup{spt}(\mu))) >0$.
\end{cor}

\begin{proof}
Since 
\[
\mathcal{J}_{d/2,1/2}(\mu)^2 = \int |\widehat \mu(z) |^4 \, dz
\]
this is an immediate consequence of Theorem \ref{maindistance2} with $\nu=\mu$.
\end{proof}

\subsection{Proof of Theorem \ref{maindistance2} }

We first prove the claim in the case when $s+t \geq d$.  Write $\nu_0$ for the measure defined by $ \nu_0(E) = \nu(-E)$ for Borel sets $E$ where $-E = \{ -x : x \in E\}$.  Then,   by Theorem \ref{convolution},
\[
\int |\widehat{\mu * \nu_0}(z)|^2 \, dz \lesssim  \mathcal{J}_{s+t,1}(\mu * \nu_0) \lesssim  \mathcal{J}_{s, \theta}(\mu ) \mathcal{J}_{t,1 -\theta}(\nu_0 ) = \mathcal{J}_{s, \theta}(\mu ) \mathcal{J}_{t,1 -\theta}(\nu  ) < \infty.
\]
It follows that $\widehat{\mu * \nu_0} \in L^2(\mathbb{R}^d)$ and therefore $\mu * \nu_0 \in L^2(\mathbb{R}^d)$ by \cite[Theorem 3.3]{mattila}.  Since $\mu * \nu_0$ is supported by the sumset $\textup{spt}(\mu)+  \textup{spt}(\nu_0) = \textup{spt}(\mu)-  \textup{spt}(\nu)$, we conclude that the \emph{difference} set  $A= \textup{spt}(\mu)-  \textup{spt}(\nu)$ has positive $d$-dimensional Lebesgue measure.  Therefore
\[
0< \int_{\rd} \textbf{1}_{A}(z) \, dz = \int_{S^{d-1}} \int_{0}^\infty \textbf{1}_{A}(rv) \, r^{d-1}dr  d\sigma_{d-1}(v) 
\]
where $\sigma_{d-1}$ is the surface measure on $S^{d-1}$.  This  implies the existence of (many) $v \in S^{d-1}$ such that $\mathcal{L}^1(\{ r : rv \in A\})>0$.  However, for all $v \in S^{d-1}$,  $\{ r : rv \in A\} \subset  D(\textup{spt}(\mu), \textup{spt}(\nu))$, proving the result.

The proof in the sub-critical case $s+t<d$ is morally similar but rather more complicated because we cannot disintegrate $s$-dimensional Hausdorff measure $\mathcal{H}^s$ for $s<d$ since it is not a $\sigma$-finite measure.  Fortunately, we can  appeal to a deep result of Mattila which connects   (quadratic) spherical averages and distance sets, see, for example,  \cite[Proposition 15.2 (b)]{mattila} and \cite[Theorems 4.16-4.17]{mattiladistance}.  Following \cite{mattiladistance}, we need a more general  statement than \cite[Proposition 15.2 (b)]{mattila} where we  associate a `distance measure' to two measures rather than a single measure (repeated twice).  Define  the mixed quadratic spherical average of $\mu$ and $\nu$ for $r>0$ by 
\[
\sigma_{\mu,\nu}(r) = \int_{S^{d-1}} \widehat \mu(r v)  \overline{\widehat \nu(r v)}\, d\sigma_{d-1}(v)
\]
noting that  $\sigma_{\mu,\mu}(r) $ is the usual quadratic spherical average from \cite[Section 15.2]{mattila}. The mutual energy of $\mu$ and $\nu$ may then be expressed as 
\[
\mathcal{I}_s(\mu,\nu) : = \int \int \frac{d \mu(x) \, d \nu (y)}{|x-y|^s} \approx    \int_{\rd} \widehat \mu(z) \overline{\widehat \nu(z)} |z|^{s-d} \, dz = \int_{0}^\infty \sigma_{\mu,\nu}(r)  r^{s-1} \, dr 
\]
for $s\in (0,d)$.  Define a `mixed distance measure' $\delta_{\mu, \nu}$ by
\[
\int \phi(x) \, d\delta_{\mu, \nu}(x) = \int \int \phi(|x-y|) \, d\mu x \, d \nu y
\]
for continuous functions $\phi$ on $\mathbb{R}$. This recovers the distance measure $\delta(\mu)$ defined in \cite{mattila} for $\mu=\nu$.  Then $\delta_{\mu, \nu}$  is a finite Borel measure supported on the mixed distance set $D(\textup{spt}(\mu), \textup{spt}(\nu))$. Finally,  define the weighted mixed distance measure $\Delta_{\mu, \nu}$ by
\[
\int \phi(x) \, d\Delta_{\mu, \nu}(x) = \int  u^{(1-d)/2} \phi(u) \, d\delta_{\mu, \nu}(u) 
\]
for continuous functions $\phi$ on $\mathbb{R}$.  The following result is from (the proof of) \cite[Theorem 4.17]{mattiladistance} and is the key technical tool in proving Theorem \ref{maindistance2}.  It   recovers \cite[Proposition 15.2 (b)]{mattila} when $\mu=\nu$.

\begin{prop}(\cite[(Proof of) Theorem 4.17]{mattiladistance}) \label{technical}
Suppose $\mu$ and $\nu$ are finite Borel measures on $\rd$ and $0 < \alpha, \beta <d$ are such that $\mathcal{I}_\alpha(\mu)< \infty$ and $\mathcal{I}_\beta(\nu)< \infty$.  If $\gamma \in (0,1)$ is such that $\gamma < \alpha+\beta-d+1$ and
\[
\int_1^\infty \sigma_{\mu,\nu}(r) ^2 r^{d+\gamma-2} \, dr < \infty,
\]
then $\mathcal{I}_\gamma(\Delta_{\mu,\nu}) < \infty$.
\end{prop}

We are ready to prove the sub-critical bound from Theorem \ref{maindistance2}. We aim to apply Proposition \ref{technical} with  $\alpha<s$ and $\beta < t$ where $s$ and $t$ are from the  statement of Theorem \ref{maindistance2}.   First, note that
\[
\mathcal{I}_\alpha(\mu) \approx \mathcal{J}_{\alpha,1}(\mu) \lesssim \mathcal{J}_{s,\theta}(\mu) < \infty
\]
and
\[
\mathcal{I}_\beta(\mu) \approx \mathcal{J}_{\beta,1}(\nu) \lesssim \mathcal{J}_{t,1-\theta}(\mu) < \infty
\]
by assumption and since the Fourier spectrum is non-decreasing, see (proof of) Theorem \ref{concave}. Choose  $\gamma \in (0,1)$   such that $ d+\gamma - 2<\alpha+\beta-1$.  Then 
\begin{align*}
& \hspace{-0.5cm} \int_1^\infty \sigma_{\mu,\nu}(r) ^2  r^{d+\gamma-2} \, dr \leq  \int_1^\infty \left( \int_{S^{d-1}} |\widehat \mu(r v)|  |\widehat \nu(r v)|\, d\sigma_{d-1}(v)\right)^2  r^{\alpha+\beta-1} \, dr \\
&\lesssim    \int_1^\infty \int_{S^{d-1}} |\widehat \mu(r v)|^2   |\widehat \nu(r v)|^2   r^{\alpha+\beta-1} \, d\sigma_{d-1}(v)\, dr \qquad \text{(Jensen's inequality)} \\
&=    \int_1^\infty \int_{S^{d-1}} |\widehat \mu(r v)|^2  r^{\alpha-\theta}   |\widehat \nu(r v)|^2   r^{\beta+\theta-1} \, d\sigma_{d-1}(v)\, dr \\
&\leq     \left(\int_1^\infty \int_{S^{d-1}} |\widehat \mu(r v)|^{2/\theta}  r^{\alpha/\theta-1}  \, d\sigma_{d-1}(v)\, dr  \right)^\theta \\
&\, \hspace{1cm} \cdot \left( \int_1^\infty \int_{S^{d-1}}   |\widehat \nu(r v)|^{2/(1-\theta)}   r^{\beta/(1-\theta) - 1} \, d\sigma_{d-1}(v)\, dr  \right)^{1-\theta}  \qquad   \text{(H\"older's inequality)}\\
&\leq  \left(\int_{\mathbb{R}^d} |\widehat \mu(z)  |^{2/\theta} |z|^{\alpha/\theta-d} \, dz\right)^{\theta}  \left(\int_{\mathbb{R}^d} |\widehat \nu(z)  |^{2/(1-\theta)} |z|^{\beta/(1-\theta)-d} \, dz\right)^{1-\theta}\\
&=  \mathcal{J}_{\alpha,\theta}(\mu)   \mathcal{J}_{\beta,1-\theta}(\nu) \\
&\leq  \mathcal{J}_{s,\theta}(\mu)   \mathcal{J}_{t,1-\theta}(\nu) \\
& < \infty.
\end{align*}
Apply Proposition \ref{technical} to deduce that 
\[
\hd D(\textup{spt}(\mu), \textup{spt}(\nu)) =  \hd \textup{spt}(\Delta_{\mu, \nu} ) \geq \gamma
\]
and then letting $\gamma \to \alpha+\beta+1-d $ and then  $\alpha \to s$ and $\beta \to t$ proves the result.

Finally, Theorem \ref{maindistance} follows easily from Theorem \ref{maindistance2}.  Since $\dim_\textup{F}^\theta X$ is continuous in $\theta \in [0,1]$ by Theorem \ref{ctyX}, we may choose $\theta \in [0,1]$ such that 
\[
 \dim_\textup{F}^\theta X + \dim_\textup{F}^{1-\theta} X
\]
attains its supremum.  Then, by definition, for all $s< \dim_\textup{F}^\theta X $ and $t<\dim_\textup{F}^{1-\theta} X $,  there exist finite Borel measures  $\mu$ and $\nu$ with   supports contained in $X$ such that 
 \[
\mathcal{J}_{s,\theta}(\mu) < \infty
\]
and
\[
\mathcal{J}_{t,1-\theta}(\nu) < \infty.
\]
The result then follows by applying Theorem \ref{maindistance2}.

\section{Random sets and measures}

There are many interesting connections between Fourier analysis and random processes, see \cite{kahane} especially. In part as an invitation to further investigation along these lines, we close by presenting two applications of the Fourier spectrum to problems concerning random sets and measures.

There is a well-known connection between Fourier decay of random measures and decay of the moments $\mathbb{E}(|\widehat \mu(z)|^{k})$, see \cite{kahane}.   In particular, a method of Kahane allows one to transfer quantitative information about the moments to an almost sure lower bound on the Fourier dimension. This technique has been used to show that many random sets are almost surely Salem, including fractional Brownian images. Kahane's approach was   formalised in a general setting by Ekstr\"om \cite[Lemma 6]{ekstrom}.  We  recover and generalise Ekstr\"om's lemma    with a very simple proof.

\begin{lma} \label{moment}
Suppose $\mu$ is a random finite Borel measure on $\rd$ such that for $\theta \in (0,1]$ and all $z \in \mathbb{R}^d$
\[
\mathbb{E}(|\widehat \mu(z)|^{2/\theta}) \lesssim |z|^{-s/\theta}.
\]
Then  $\fs \mu \geq s$ almost surely. In particular, if  $\widehat \mu$ is H\"older almost surely and for a sequence of  $k \in \mathbb{N}$ tending to infinity  it holds that for all $z \in \mathbb{R}^d$
\[
\mathbb{E}(|\widehat \mu(z)|^{2k}) \lesssim |z|^{-sk},
\]
 then $\fd \mu \geq s$ almost surely. Further, if    $\widehat \mu$ is $\alpha$-H\"older almost surely and there exists $\eps>0$ and $\theta \in (0,1)$ such that
\[
\mathbb{E}(|\widehat \mu(z)|^{2/\theta}) \lesssim |z|^{-(d+\eps)}
\]
 then 
\[
\fd \mu \geq \frac{2 \alpha \eps \theta}{2 \alpha+ d\theta }>0
\] 
almost surely.
\end{lma}

\begin{proof}
Let $t < s$.  By Fubini's theorem
\begin{align*}
\mathbb{E} ( \mathcal{J}_{t,\theta}(\mu)^{1/\theta}) \leq \int_{\rd} |z|^{t/\theta-d} \mathbb{E}(|\widehat \mu(z)|^{2/\theta}) \, dz \lesssim \int_{\rd} |z|^{(t-s)/\theta-d} \, dz < \infty
\end{align*}
which proves the first claim. The second claim then follows since $\widehat \mu$ is almost surely continuous and the third claim follows   by combining the first claim with the bounds from Theorem \ref{cty0}. 
\end{proof}

To bound the Fourier dimension from below using Lemma \ref{moment}, we only require a sufficiently good bound for a single moment, whereas   \cite[Lemma 6]{ekstrom} (and \cite[Lemma 1, page 252]{kahane} which its proof relies on) requires control of all the moments simultaneously. Moreover, \cite[Lemma 6]{ekstrom} requires the measures to be almost surely supported in a fixed compact set, whereas our random measures can be unbounded as long as $\widehat \mu$ is H\"older almost surely (with a random H\"older exponent). For both of these improvements, our gain comes from continuity of the Fourier spectrum. Among other things, Ekstr\"om uses the above approach to construct random images of compactly supported measures with large Fourier dimension almost surely \cite[Theorem 2]{ekstrom}.  Using Lemma \ref{moment} this can be extended to include   non-compactly supported measures. We leave the details to the reader.

For our second application we consider an explicit random model based on fractional Brownian motion in a setting where the Fourier dimension alone yields only trivial results.  Given $\alpha \in (0,1)$ and integers $n,d \geq 1$, let $B^\alpha: \mathbb{R}^n \to \rd$ be index $\alpha$ fractional Brownian motion.  See \cite[Chapter 18]{kahane} for the construction and a detailed analysis of this process ($B^\alpha$ is Kahane's $(n,d,\gamma)$ Gaussian process with $\gamma=2\alpha$).  In particular, \cite[page 267, Theorem 1]{kahane} gives that for a compact set $Y \subseteq \mathbb{R}^n$ with $\hd Y > s$ the image  $B^\alpha(Y)$ almost surely  supports a measure $\mu$ satisfying
\begin{equation} \label{kahaneeq}
\sd \mu \geq \fd \mu \geq  s/\alpha.
\end{equation}
Combining this with Theorem \ref{convolution} immediately gives the following.

\begin{cor} \label{randomeasy}
Let $X \subseteq  \rd$ and $Y \subseteq \mathbb{R}^n$ be  non-empty Borel sets.   If $\hd Y >0$ and
\[
\alpha < \frac{\hd Y}{d-\hd X},
\]
then $X+ B^\alpha(Y)$ has positive $d$-dimensional Lebesgue measure almost surely.  Otherwise, 
\[
\hd (X+ B^\alpha(Y)) \geq  \hd X + \frac{\hd Y}{\alpha}
\]
almost surely.
\end{cor}

Crucial to the above result is that the image $B^\alpha(Y)$ is a Salem set almost surely. We are interested in the following variant where this is no longer the case.  Let $V$ be a $k$-dimensional subspace of $\rd$ for integers $1 \leq k < d$ and $X\subseteq \rd$ and $Y \subseteq \mathbb{R}^n$ be  non-empty Borel sets. Let $B^\alpha: \mathbb{R}^n \to V$ be index $\alpha$ fractional Brownian motion with $V$ identified with $\mathbb{R}^k$.  We are interested in almost sure lower bounds for the dimension of $X+ B^\alpha(Y)$.  However,  this time  $B^\alpha(Y)$ has Fourier dimension 0 and so non-trivial estimates are not possible using only the Fourier dimension. If $\hd X \geq k$, then we cannot improve on the trivial estimate $\hd (X+ B^\alpha(Y)) \geq \hd X$ since, for example, $X$ could be a subset of $V$ with dimension $k$.  However, if $\hd X < k$, then we can use the Fourier spectrum to derive non-trivial bounds.  Moreover, no matter what values of $\alpha$,  $\hd X$ and $\hd Y$ we assume, we can never conclude that $X+ B^\alpha(Y)$ has positive $d$-dimensional Lebesgue measure in general.  For example, consider $X \subseteq V \times F$ for a set $F\subseteq V^\perp$ with zero  $(d-k)$-dimensional Lebesgue measure.  In this case $X+ B^\alpha(Y) \subseteq V \times F$   has zero $d$-dimensional Lebesgue measure.

\begin{cor} \label{randomsum}
Let $V$ be a $k$-dimensional subspace of $\rd$ for integers $1 \leq k < d$ and $X\subseteq \rd$ and $Y \subseteq \mathbb{R}^n$ be  non-empty Borel sets. Let $B^\alpha: \mathbb{R}^n \to V$ be index $\alpha$ fractional Brownian motion with $V$ identified with $\mathbb{R}^k$.   If $\hd X < k$, then almost surely
\[
\hd (X+ B^\alpha(Y)) \geq   \min\left\{ \hd X + \frac{\hd Y}{\alpha}-\frac{\hd X \hd Y}{k\alpha} , k\right\}.
\]
\end{cor}

\begin{proof}
By Theorem \ref{convolution}, the definition of the Fourier spectrum for sets, \eqref{kahaneeq} and Theorem \ref{embedding},
\begin{align*}
\hd (X+ B^\alpha(Y))  &\geq \min\left\{  \sup_{\theta \in (0,1)} \left(\dim_\textup{F}^{1-\theta} X + \fs B^\alpha(Y) \right), \ d \right\} \\
&\geq \sup_{\theta \in (0,1)} \left( (1-\theta) \hd X + \min\left\{ k \theta , \frac{\hd Y}{\alpha}\right\} \right) \\
&= \min\left\{ \hd X + \frac{\hd Y}{\alpha}-\frac{\hd X \hd Y}{k\alpha} , k\right\}
\end{align*}
as required.
\end{proof}

In the above proof, we were forced to make the `worst case' estimate $\dim_\textup{F}^{1-\theta} X \geq (1-\theta) \hd X$ since we made no assumptions about $X$.  Clear improvements are possible if we make assumptions about the  Fourier dimension  spectrum of $X$ (for example, if $X$ is Salem) but we leave the details to the reader.  We do not know if the almost sure lower bounds given above are the best possible in general. 
\begin{ques}
Is the almost sure  lower bound  from Corollary \ref{randomsum} sharp?
\end{ques}
The bounds are sharp for $\alpha \leq \hd Y/ k$ as the following example will show.  Let $X = E \times F$ where $E \subseteq V$  and $F \subseteq V^\perp$ with $\hd F =0$.   It is possible to arrange for $\hd X = \pd E$ and for this to assume any value in the interval $[0,k]$.  Here $\pd$ is again the packing dimension.  Let $Y \subseteq \mathbb{R}^n$ satisfy $\hd Y = \pd Y$.   Then $X+B^\alpha(Y) = (E+B^\alpha(Y)) \times F$ and (for all realisations of $B^\alpha$)
\begin{align*}
\hd (X+B^\alpha(Y)) \leq \pd (E+B^\alpha(Y)) + \hd F &=  \pd (E+B^\alpha(Y)) \\ 
&\leq \min\left\{ \pd E+\frac{\pd Y}{\alpha}, \, k\right\} \\ 
&= \min\left\{ \hd X+\frac{\hd Y}{\alpha}, \, k\right\}
\end{align*}
which coincides with the general almost sure lower bound for $\alpha \leq \hd Y / k$.

One might hope that this problem could be reduced to the simpler setting of Corollary \ref{randomeasy} by decomposing $X+ B^\alpha(Y)$ into slices parallel to $V$.  We have some doubt over this approach due to the following example. Consider the case where  $X$ is a `graph' in the sense that $X \subseteq V \times F$ with the property that $ X \cap (V+x)$ is a single point for all $x \in F$.  Despite the minimal fibre structure, there is no restriction on $\hd X$ and it can take any value in the interval $[\hd F,\hd F+ k]$.  Then $(X+B^\alpha(Y)) \cap  (V+x)$ is a translation of $B^\alpha(Y)$  for all $x \in F$ and estimating the dimension of $X+B^\alpha(Y)$ using standard slicing methods, e.g. Marstrand's slice theorem, see \cite{mattila, falconer}, cannot do better than
\[
\hd (X+ B^\alpha(Y))  \geq \min\left\{  \frac{\hd Y}{\alpha}, \, k\right\}  + \hd F
\]
almost surely.  For $\alpha > \hd Y/k$ this estimate can be much poorer than the estimate from Corollary \ref{randomsum} and can even be  worse than the trivial lower bound $\hd X$.

\section*{Acknowledgements}

I am grateful to Pertti Mattila,  Amlan Banaji,  Natalia Jurga, St\'ephane Seuret, Dingkang Ren and Ana de Orellana  for helpful comments.

\section*{Conflict of interest statement}

On behalf of all authors, the corresponding author states that there is no conflict of interest.

\section*{Data availability statement}

There is no data associated with this manuscript.

\end{document}